\documentclass[12pt]{amsart}

\usepackage{latexsym, url}
\usepackage{amsthm}
\usepackage{amsmath}
\usepackage{amsfonts}
\usepackage{amssymb}

\usepackage{mathabx} 
\usepackage{nameref} 

\usepackage{hyperref}

\usepackage{quiver}

\input xy
\xyoption{all}

\usepackage{url}
\usepackage{tikz-cd}
\usetikzlibrary{babel}
\usepackage{fancyhdr}
\usepackage{bm}
\usepackage{caption}
\usepackage{mathtools}
\usepackage{xcolor}
\usepackage{paralist}
\usepackage{opakuj}
\usepackage{graphicx}
\usepackage{booktabs}
\usepackage{graphicx}
\usepackage{tabu}
\usepackage{multirow}

\usepackage{titletoc}
\usepackage[normalem]{ulem}

\theoremstyle{definition}
\newtheorem{theo}{Theorem}[section]

\newtheorem{lemma}[theo]{Lemma}
\newtheorem{prop}[theo]{Proposition}
\newtheorem{defi}[theo]{Definition}

\newtheorem{pozn}[theo]{Remark}

\newtheorem{con}[theo]{Construction}
\newtheorem{nota}[theo]{Notation}

\newtheorem{pr}[theo]{Example}
\newtheorem{nonpr}[theo]{Non-example}

\newcommand{\talgl}{\text{T-Alg}_l}
\newcommand{\talgs}{\text{T-Alg}_s}
\newcommand{\talgc}{\text{T-Alg}_c}

\newcommand\res[1]{\text{Res}(#1)}

\newcommand\cnr[1]{\text{Cnr}(#1)}

\newcommand\cod[1]{\text{CoDesc}(#1)}

\newcommand{\gph}{\text{Graph}}
\newcommand{\fc}{\text{fc}}
\newcommand{\set}{\text{Set}}

\newcommand\spn[1]{\text{Span}(#1)}

\newcommand{\cat}{\text{Cat}}
\newcommand\dbl{\text{Dbl}}
\newcommand\ob{\text{ob }}
\newcommand\mor{\text{mor }}

\newcommand\spanc{\text{Span}}

\newcommand\sq[1]{\text{Sq}(#1)}
\newcommand\pbsq[1]{\text{PbSq}(#1)}

\newcommand{\pscr}{\text{Crossed}}

\newcommand{\sfs}{\mathcal {SFS}}
\newcommand{\ofs}{\mathcal {OFS}}
\newcommand{\coddiscr}{\text{CodDiscr}}
\newcommand{\factdbl}{\text{FactDbl}}

\newcommand\ca{\mathcal {A}}
\newcommand\cb{\mathcal {B}}
\newcommand\cc{\mathcal {C}}

\newcommand\cf{\mathcal {F}}

\newcommand\ce{\mathcal {E}}

\newcommand\ck{\mathcal {K}}

\newcommand\cm{\mathcal {M}}

\newcommand\cs{\mathcal {S}}

\newcommand\cy{\mathcal {Y}}

\newcommand\cz{\mathcal {Z}}

\definecolor{arsenic}{rgb}{0.23, 0.27, 0.29}
\definecolor{ashgrey}{rgb}{0.7, 0.75, 0.71}
\definecolor{charcoal}{rgb}{0.12, 0.18, 0.22}
\definecolor{bulgarianrose}{rgb}{0.28, 0.02, 0.03}
\definecolor{darkblue}{rgb}{0, 0, 0.5}

\date{\today}

\newcommand\blfootnote[1]{%
  \begingroup
  \renewcommand\thefootnote{}\footnote{#1}%
  \addtocounter{footnote}{-1}%
  \endgroup
}

\begin{document}
\title[Factorization systems and double categories]
{Factorization systems and double categories}
\author[Miloslav Štěpán]
{Miloslav Štěpán}

\makeatletter
\dottedcontents{section}[1.5em]{\bfseries}{1.3em}{.6em}
\makeatother

\begin{abstract}
We show that factorization systems, both strict and orthogonal, can be equivalently described as double categories satisfying certain properties. This provides conceptual reasons for why the category of sets and partial maps or the category of small categories and cofunctors admit orthogonal factorization systems.

The theory also gives an explicit description of various lax morphism classifiers and explains why they admit strict factorization systems.


\end{abstract} 

\blfootnote{This work was supported from Operational Programme Research, Development and Education ``Project Internal Grant Agency of Masaryk University'' (No. $\text{CZ}.02.2.69 \symbol{92} 0.0\symbol{92} 0.0\symbol{92}19 \underline{\enskip} 073\symbol{92}0016943)$.\\
The author also acknowledges the support of Grant Agency of the Czech Republic under the grant 22-02964S.}

\maketitle
\tableofcontents

\section{Introduction}

Both double categories and factorization systems are widespread in category theory. In this paper, we will study the relationship between them. A double category consists of objects, vertical morphisms, horizontal morphisms and squares just like the one pictured below:
\[\begin{tikzcd}
	a && b \\
	\\
	c && d
	\arrow[""{name=0, anchor=center, inner sep=0}, "g", from=1-1, to=1-3]
	\arrow["v", from=1-3, to=3-3]
	\arrow["u"', from=1-1, to=3-1]
	\arrow[""{name=1, anchor=center, inner sep=0}, "h"', from=3-1, to=3-3]
	\arrow["\alpha", shorten <=9pt, shorten >=9pt, Rightarrow, from=0, to=1]
\end{tikzcd}\]
In a general double category you cannot compose vertical morphisms with horizontal ones, but if you could, you might interpret the above square $\alpha$ as telling us that the morphism $v \circ g$ (horizontal followed by vertical) can be factored as $h \circ u$ (vertical followed by horizontal) - this is reminiscent of ordinary factorization systems on a category.

Taking this philosophy to heart, we assign to a double category $X$ a certain \textit{category of corners} $\cnr{X}$ (a concept introduced by Mark Weber in \cite{internalalg}), in which composition of vertical and horizontal morphisms is possible, and for which squares in $X$ turn into commutative squares in $\cnr{X}$. Regarding the double category as a diagram $X: \Delta^{op} \to \cat$, producing $\cnr{X}$ amounts to taking the \textit{codescent object} of $X$, a 2-categorical colimit that is an analogue of ordinary coequalizers.

Because the usage of codescent objects is related to the classification of pseudo/lax morphisms between strict algebras for a 2-monad (see \cite{codobjcoh}), the category of corners construction gives us an explicit description of lax morphism classifiers for a certain class of 2-monads.

The paper is organized as follows:

\begin{itemize}
\item In Section \ref{sekce_dblcats} we recall the basic notions of double category theory and define a slight generalization of crossed double categories of   \cite{internalalg}. We also describe the category of corners construction for this class of double categories and mention some examples.

\item In Section \ref{sekcefsdbl} we establish two equivalences: the first is the equivalence between strict factorization systems and double categories for which every top right corner can be uniquely filled into a square:
\[\begin{tikzcd}
	a & b \\
	& c
	\arrow["g", from=1-1, to=1-2]
	\arrow["u", from=1-2, to=2-2]
\end{tikzcd}\]
The second is the equivalence between orthogonal factorization systems and a special kind of crossed double categories.
 
\item In Section \ref{sekcecodobjdbl} we recall codescent objects and use the category of corners construction to give explicit descriptions for various (co)lax morphism classifiers.
\end{itemize}

\noindent \textbf{Prerequisities}: We assume basic familiarity with double categories and factorization systems. For Section \ref{sekcecodobjdbl} we further assume that the reader is familiar with lax/colax morphisms between strict $T$-algebras for a 2-monad, as defined in \cite[Page 1.2]{twodim}.

\bigskip

\noindent \textbf{Acknowledgements}: I want to thank my Ph.D. supervisor John Bourke for his guidance and careful readings of all the drafts of this paper.

\section{Double categories}\label{sekce_dblcats}



\subsection{Basic notions}\label{subsekce_dblcats}

\begin{nota}
A \textit{double category} $X$ is an internal category in $\cat$. In particular it consists of the following diagram in $\cat$:
\[\begin{tikzcd}
	{X_2} && {X_1} && {X_0}
	\arrow["s"{description}, from=1-5, to=1-3]
	\arrow["{d_1}"{description}, shift left=4, from=1-3, to=1-5]
	\arrow["{d_0}"{description}, shift right=4, from=1-3, to=1-5]
	\arrow["{d_1}"{description}, from=1-1, to=1-3]
	\arrow["{d_0}"{description}, shift right=4, from=1-1, to=1-3]
	\arrow["{d_2}"{description}, shift left=4, from=1-1, to=1-3]
\end{tikzcd}\]
We call objects of $X_0$ \textit{the objects} of $X$, morphisms of $X_0$ the \textit{vertical morphisms}, objects of $X_1$ the \textit{horizontal morphisms}, the morphisms of $X_1$ \textit{squares}:
\[\begin{tikzcd}
	a && b \\
	\\
	c && d
	\arrow["u"', from=1-1, to=3-1]
	\arrow[""{name=0, anchor=center, inner sep=0}, "g", from=1-1, to=1-3]
	\arrow["v", from=1-3, to=3-3]
	\arrow[""{name=1, anchor=center, inner sep=0}, "h"', from=3-1, to=3-3]
	\arrow["\alpha", shorten <=9pt, shorten >=9pt, Rightarrow, from=0, to=1]
\end{tikzcd}\]
We refer to the composition in the category $X_1$ as \textit{vertical square composition}, the \textit{horizontal} composition of horizontal morphisms as well as squares is given by the functor $d_1: X_2 \to X_1$. Similarly, an identity morphism in $X_1$ will be called a \textit{vertical identity square}, and a \textit{horizontal identity square} on a vertical morphism $u$ will be given by $s(u)$. By an \textit{identity square} in $X$ we mean a square that is either a vertical or a horizontal identity. Objects and horizontal morphisms form a category that we'll denote by $h(X)$.

Double categories together with \textit{double functors} form a category that we'll denote by $\dbl$.
\end{nota}

Here we recall some definitions from \cite{limitsindbl} that will make an appearance in the paper:

\begin{defi}[Duals and transposes]
Any double category $X$ has its \textit{transpose} $X^T$ obtained by switching vertical and horizontal morphisms and compositions. Any double category $X$ has its \textit{vertical dual} which we denote by $X^{v}$. It is  defined by putting:
\[
(X^{v})_i = X_i^{op} \text{ for } i \in \{ 0,1,2 \}.
\]
Similarly, domain and codomain functors for $X^{v}$ are obtained by applying $(-)^{op}$ on those for $X$.

There is also a notion of a \textit{horizontal dual} $X^{h}$, which is a diagram obtained from $X$ by switching $d_0$'s and $d_1$'s.
\end{defi}

\begin{defi}
A double category $X$ is \textit{flat} if any square $\alpha$ is uniquely determined by its boundary.
\end{defi}

\begin{defi}
A double category $X$ is (strictly) \textit{horizontally invariant} if for any two invertible horizontal morphisms $g,h$ and every vertical morphism $u$ there exists a unique square filling the picture\footnote{This definition differs from the one in \cite{limitsindbl} in that we require the filler square to be unique.}:
\[\begin{tikzcd}
	a && b \\
	\\
	c && d
	\arrow[""{name=0, anchor=center, inner sep=0}, "{g \cong}", from=1-1, to=1-3]
	\arrow["u", from=1-3, to=3-3]
	\arrow[""{name=1, anchor=center, inner sep=0}, "h\cong"', from=3-1, to=3-3]
	\arrow[dashed, from=1-1, to=3-1]
	\arrow["{\exists ! \lambda}", shorten <=9pt, shorten >=9pt, Rightarrow, from=0, to=1]
\end{tikzcd}\]

\noindent We say that $X$ is (strictly) \textit{vertically invariant} if its transpose $X^T$ is horizontally invariant. A double category that is both horizontally and vertically invariant will be called (strictly) \textit{invariant}.
\end{defi}

\color{black}

\begin{pr}\label{prsqcc}
Let $\cc$ be a category. There is a double category $\sq{\cc}$ such that:
\begin{itemize}
\item objects are the objects of $\cc$,
\item vertical morphisms are those of $\cc$,
\item horizontal morphisms are those of $\cc$,
\item squares are commutative squares in $\cc$.
\end{itemize}
\end{pr}

\begin{pr}\label{prpbsqmpbsq}
There is a sub-double category $\text{PbSq}(\cc) \subseteq \sq{\cc}$ with the same objects and morphisms, whose squares are the pullback squares in $\cc$.

Another example is a sub-double category $\text{MPbSq}(\cc) \subseteq \text{PbSq}(\cc)$ with the same objects and horizontal morphisms whose vertical morphisms are monomorphisms in the category $\cc$.
\end{pr}

\begin{pr}\label{prbofib}
Let $\ce$ be a category with pullbacks. There is a double category $\text{BOFib}(\ce)$ such that:
\begin{itemize}
\item objects are internal categories\footnote{As defined in \cite[8.1]{handbook} for instance.} in $\ce$,
\item vertical morphisms are internal functors that are \textit{discrete opfibrations}, that is, internal functors $F: A \to B$ such that the following square is a pullback in $\ce$:
\[\begin{tikzcd}
	{A_1} && {B_1} \\
	\\
	{A_0} && {B_0}
	\arrow["s", from=1-3, to=3-3]
	\arrow["s"', from=1-1, to=3-1]
	\arrow["{F_1}", from=1-1, to=1-3]
	\arrow["{F_0}"', from=3-1, to=3-3]
	\arrow["\lrcorner"{anchor=center, pos=0.125}, draw=none, from=1-1, to=3-3]
\end{tikzcd}\]
\item horizontal morphisms are internal functors $F: A \to B$ that are \textit{bijections on objects}, i.e. the object part morphism $F_0 : A_0 \to B_0$ is an isomorphism in $\ce$,
\item a square in $\text{BOFib}(\ce)$ is a pullback square in $\cat(\ce)$.
\end{itemize}
\end{pr}

\color{black}

\noindent Note that all of the above examples are flat and invariant.

\subsection{Crossed double categories}\label{subsekce_pscr}


Crossed double categories (a generalization of \textit{crossed simplicial groups} of \cite{crossedsimpgps}) were introduced by Mark Weber in \cite{internalalg} to calculate various internal algebra classifiers. For instance, if $S$ is the free symmetric strict monoidal category 2-monad on $\cat$, the bar construction (also called a \textit{resolution}, see Definition \ref{defi_resolution}) $\res{*}$ of a terminal $S$-algebra $*$ has the structure of a crossed double category.

In that paper, any crossed double category can be turned into a category ``in the best possible way” - this is called the \textit{category of corners} construction. In the case of $\res{*}$ the category of corners construction produces the \textit{free symmetric strict monoidal category containing a commutative monoid}, which happens to be the category $\text{FinSet}$ of finite ordinals and all functions between them.

In this paper we consider a slight generalization of crossed double categories, obtained by dropping the ``splitness” assumption on the opfibration that appears in the definition. This allows us to consider a bigger class of examples - the ones for which there is no canonical choice of ``opcartesian lifts”. We then present an analogue of the category of corners construction for this wider class of double categories and prove some of its key properties. All of this is in preparation for Section \ref{sekcefsdbl} where we show that under some conditions the category of corners admits a strict or an orthogonal factorization system.




\subsubsection{Definition and examples}

\begin{defi}\label{deficrossed}
A double category $X$ is said to be \textit{crossed} if $d_0 : X_1 \to X_0$ is an opfibration and $d_1: X_2 \to X_1, s : X_0 \to X_1$ are morphisms of opfibrations\footnote{Note that the map $d_0^2 = d_0 \circ d_0 : X_2 \to X_0$ is an opfibration since it is a composite of an opfibration and a pullback of an opfibration.}:
\[\begin{tikzcd}
	{X_0} && {X_1} && {X_2} \\
	\\
	&& {X_0}
	\arrow["{s_0}", from=1-1, to=1-3]
	\arrow["{d_1}"', from=1-5, to=1-3]
	\arrow["{d_0}", from=1-3, to=3-3]
	\arrow["{d_0^2}", from=1-5, to=3-3]
	\arrow[Rightarrow, no head, from=1-1, to=3-3]
\end{tikzcd}\]
\end{defi}

\noindent In elementary terms, this is to say the following:

A square $\kappa$ is said to be \textit{opcartesian} (by which we mean $d_0$-opcartesian if regarded as a morphism in $X_1$) when given any square $\alpha$ and a vertical morphism $v$ (as in the picture below), there exists a unique square $\beta$ so that the following equality of squares holds:
\begin{equation}\label{eqopcartesiansquare}
\begin{tikzcd}
	a && b && a && b \\
	{\widehat{b}} && c & {=} &&& c \\
	d && c && d && c
	\arrow[""{name=0, anchor=center, inner sep=0}, "g", from=1-1, to=1-3]
	\arrow["u", from=1-3, to=2-3]
	\arrow["{\widehat{u}}"', from=1-1, to=2-1]
	\arrow[""{name=1, anchor=center, inner sep=0}, "{\widehat{g}}"{description}, from=2-1, to=2-3]
	\arrow["v", from=2-3, to=3-3]
	\arrow["w"', from=1-5, to=3-5]
	\arrow[""{name=2, anchor=center, inner sep=0}, "g", from=1-5, to=1-7]
	\arrow[""{name=3, anchor=center, inner sep=0}, "h"', from=3-5, to=3-7]
	\arrow[""{name=4, anchor=center, inner sep=0}, "h"', from=3-1, to=3-3]
	\arrow["\theta"', from=2-1, to=3-1]
	\arrow["u", from=1-7, to=2-7]
	\arrow["w"', curve={height=24pt}, from=1-1, to=3-1]
	\arrow["v", from=2-7, to=3-7]
	\arrow["\alpha", shorten <=9pt, shorten >=9pt, Rightarrow, from=2, to=3]
	\arrow["{\exists ! \beta}", shorten <=4pt, shorten >=4pt, Rightarrow, from=1, to=4]
	\arrow["\kappa", shorten <=4pt, shorten >=4pt, Rightarrow, from=0, to=1]
\end{tikzcd}
\end{equation}

\noindent To say that $d_0: X_1 \to X_0$ is an opfibration is to say that any tuple $(g,f)$ of a ``composable” pair of a  horizontal and a vertical morphisms (as pictured below) can be filled to an opcartesian square:
\begin{equation}\label{eqd0opfib}
\begin{tikzcd}
	a && b \\
	\\
	&& c
	\arrow["g", from=1-1, to=1-3]
	\arrow["f", from=1-3, to=3-3]
\end{tikzcd}
\end{equation}
Such tuples will be referred to as (top-right) \textit{corners}.

Finally, to say that $d_1: X_2 \to X_1$ and $s:X_0 \to X_1$ are morphisms of opfibrations is to say that the opcartesian squares are closed under horizontal composition and that every horizontal identity square is opcartesian.

We will denote by $\pscr$ the full subcategory of $\dbl$ spanned by crossed double categories.

\begin{pozn}[Split version]
Crossed double categories studied by Mark Weber (\cite{internalalg}) are defined as in Definition \ref{deficrossed}, except it is required that $d_0: X_1 \to X_0$ is a \textbf{split} opfibration and the maps $d_1, s$ are morphisms of split opfibrations.

This amounts to, for every top-right corner $(g,f)$, having \textbf{a choice} of opcartesian square $\kappa_{g,f}$ filling the corner, and requiring that identity squares are chosen opcartesian, and moreover vertical and horizontal composition of chosen opcartesian squares is chosen opcartesian. In this paper we will call them \textit{split crossed} to emphasize the presence of chosen filler squares.
\end{pozn}

\begin{pozn}[Dual version]\label{pozn_pscocr}
There is a dual version of a crossed double category that we will call \textit{co-crossed}, it is obtained by replacing ``opfibration” by ``fibration” everywhere. Note then that the double category $X$ is co-crossed if and only if $X^{v}$ is crossed.
\end{pozn}

\begin{pozn}
The requirement that every corner can be filled into an opcartesian square is equivalent to saying that every corner can be filled into a pre-opcartesian square and vertical composition of pre-opcartesian squares is pre-opcartesian. By a \textit{pre-opcartesian} square we mean a square satisfying \eqref{eqopcartesiansquare} only for squares for which $v$ is the identity. This equivalence is proven \cite[Proposition 8.1.7]{handbook2} for a general fibration.
\end{pozn}

\color{black}



\begin{pr}
If $\cc$ is a category with pullbacks, the double category $\sq{\cc}$ is co-crossed: a square is cartesian (with respect to the codomain functor\\$d_0: \cc^2 \to \cc$) if and only if it is a pullback square in $\cc$. Clearly, vertical and horizontal composition of pullback squares yields a pullback square, and identity squares are pullbacks.

For similar reasons, the following double categories are all  co-crossed. In each of these, every square is cartesian:
\begin{itemize}
\item $\text{PbSq}(\cc)$,
\item $\text{MPbSq}(\cc)$,
\item $\text{BOFib}(\ce)$.
\end{itemize}
\end{pr}

\begin{pr}\label{pr_awfs_dblcat}
Let $(L, R)$ be an \textit{algebraic weak factorization system} (\cite[2.2]{awfs}) on a category $\cc$ with pullbacks. There is an associated double category $R\text{-}\mathbb{A}lg$ of $R$-algebras. It can be shown that the codomain functor of this double category is a fibration (\cite[Proposition 8]{awfs}) and moreover, $d_1, s$ are morphisms of fibrations. Thus $R\text{-}\mathbb{A}lg$ is a co-crossed double category.
\end{pr}

Of note is that it is not these examples that will play a  role later on, but rather their vertical duals (that are crossed).


\begin{pr}\label{pr_crosseddblcats}
Assume $(T,m,i)$ is a 2-monad on a 2-category $\ck$ and let $(A,a)$ be a strict $T$-algebra. By its \textit{resolution}, denoted $\res{A,a}$, we mean the following diagram in $\talgs$:
\[\begin{tikzcd}
	{T^3A} && {T^2A} && TA
	\arrow["{Ti_A}"{description}, from=1-5, to=1-3]
	\arrow["Ta"{description}, shift right=4, from=1-3, to=1-5]
	\arrow["{m_A}"{description}, shift left=4, from=1-3, to=1-5]
	\arrow["{m_{T^2A}}"{description}, shift left=4, from=1-1, to=1-3]
	\arrow["{Tm_{TA}}"{description}, from=1-1, to=1-3]
	\arrow["{T^2a}"{description}, shift right=4, from=1-1, to=1-3]
\end{tikzcd}\]
For a cartesian\footnote{$T$ preserves pullbacks and the naturality squares for $m,i$ are pullbacks} 2-monad $T$, $\res{A,a}$ is a category internal in $\talgs$.

Let now $S$ be the free symmetric strict monoidal category 2-monad on $\cat$. Recall that for a category $\ca$, $S\ca$ has objects the tuples of objects of $\ca$, and a morphism $(a_1, \dots, a_n) \to (b_1, \dots, b_n)$ is a tuple $((f_1,\dots, f_n), \rho)$, where $\rho$ is a permutation on the n-element set and $f_i : a_i \to a_{\rho(i)}$ are morphisms in $\ca$.

Denote by $*$ the terminal $S$-algebra. Since $S$ is cartesian, $\res{*}$ is a double category. Moreover, it is (split) crossed by \cite[Example 4.4.5]{internalalg}. This double category has finite ordinals as objects, order-preserving maps as horizontal morphisms, permutations as vertical morphisms and squares being commutative squares in $\set$.

Analogous results hold if we instead consider the free braided strict monoidal category 2-monad on $\cat$. 
\end{pr}
\color{black}

\noindent A special class of a crossed double category will be of interest to us:

\subsubsection{Codomain-discrete double categories}

\begin{defi}\label{deficoddiscr}
A double category $X$ will be called \textit{codomain-discrete} if every top-right corner can be uniquely filled into a square: 
\[\begin{tikzcd}
	a && b \\
	\\
	\bullet && c
	\arrow[""{name=0, anchor=center, inner sep=0}, "g", from=1-1, to=1-3]
	\arrow["u", from=1-3, to=3-3]
	\arrow[dashed, from=1-1, to=3-1]
	\arrow[""{name=1, anchor=center, inner sep=0}, dashed, from=3-1, to=3-3]
	\arrow["{\exists !}", shorten <=9pt, shorten >=9pt, Rightarrow, from=0, to=1]
\end{tikzcd}\]
\end{defi}

This property amounts to the codomain functor $d_0: X_1 \to X_0$ being a discrete opfibration. In that case, $d_1, s$ are automatically morphisms of opfibrations and thus any codomain-discrete double category is crossed.

We denote by $\coddiscr \subseteq \dbl$ the full subcategory spanned by codomain-discrete double categories.

\begin{pozn}
Codomain-discrete double categories first appeared in  \cite[2.3]{crossedsimpgps} as double categories satisfying the \textit{star condition}.
\end{pozn}

Of note is the fact that every codomain-discrete double category is flat but not necessarily invariant as the following example demonstrates:

\color{black}
\begin{pr}\label{pr_atimesbdblcat}
Let $\ca, \cb$ be categories. There is a double category $X_{\ca,\cb}$ such that:
\begin{itemize}
\item Objects are the objects of $\ca \times \cb$,
\item vertical morphisms are morphisms in $\ca \times \cb$ of form $(f,1_b)$.
\item horizontal morphisms are morphisms in $\ca \times \cb$ of form $(1_a,g)$,
\item a square is a commutative square in $\ca \times \cb$.
\end{itemize}
This double category is clearly codomain-discrete and thus flat. It is not invariant because for example if we have $\theta, \psi$ distinct isomorphisms in $\ca$, the following can't be filled into a square:
\[\begin{tikzcd}
	{(a,b)} && {(a',b)} \\
	\\
	{(a,b')} && {(a',b')}
	\arrow["{(\theta,1_b)}", from=1-1, to=1-3]
	\arrow["{(1_{a'},f)}", from=1-3, to=3-3]
	\arrow["{(\psi,1_{b'})}"', from=3-1, to=3-3]
\end{tikzcd}\]
\end{pr}

\begin{pr}
Given a 2-monad $T$ on $\cat$ of form $\cat(T')$ for $T'$ a cartesian monad on $\set$, the transpose of the resolution $\res{A,a}$ of a strict $T$-algebra is codomain-discrete. We will encounter this class of examples in Section \ref{subsekce_laxmors}.
\end{pr}
\color{black}

\subsection{The category of corners}\label{subsekce_catofcnrs}\vphantom{.}\\

In \cite{internalalg}, given a crossed double category $X$, the category of corners $\cnr{X}$ is constructed in two steps: first, a 2-category of corners $\cb$ is constructed, and then the category of corners is obtained by taking connected components in each hom category of $\cb$, i.e. $\cnr{X} = (\pi_0)_* \cb$. In order to avoid very long proofs, we define $\cnr{X}$ for our notion of a crossed double category $X$ straight away without ever introducing $\cb$ (which in the absence of the splitness assumption would only be a weak 2-category). The universal property of this construction will be studied in Subsection \ref{subsekce_codobjdbl}.


\subsubsection{Definition and examples}

\begin{defi}\label{defi2cell}
Let $X$ be a double category and assume we're given two bottom-left corners $(e,m), (e',m')$ for which the domains of $e,e'$ and codomains of $m,m'$ agree (as pictured below). A \textit{2-cell} $\beta$ between them (denoted $\beta: (e,m) \Rightarrow (e',m')$) is a square as below for which $d_1(\beta) \circ e = e'$:
\[\begin{tikzcd}
	a \\
	{a'} && b \\
	{a''} && b
	\arrow["e"', from=1-1, to=2-1]
	\arrow["\theta"', from=2-1, to=3-1]
	\arrow[Rightarrow, no head, from=2-3, to=3-3]
	\arrow["{e'}"', curve={height=18pt}, from=1-1, to=3-1]
	\arrow[""{name=0, anchor=center, inner sep=0}, "{m'}"', from=3-1, to=3-3]
	\arrow[""{name=1, anchor=center, inner sep=0}, "m", from=2-1, to=2-3]
	\arrow["\beta", shorten <=4pt, shorten >=4pt, Rightarrow, from=1, to=0]
\end{tikzcd}\]
\end{defi}

\begin{con}\label{concnrpbsqs}
Let $X$ be a crossed double category. Define the  \textit{category of corners} $\cnr{X}$ as follows. Its objects are the objects of $X$, while the morphism $a \to b$ is an equivalence class of corners, denoted $[u, g]:a \to b$:
\[\begin{tikzcd}
	a \\
	{a'} & b
	\arrow["u"', from=1-1, to=2-1]
	\arrow["g"', from=2-1, to=2-2]
\end{tikzcd}\]
Here two corners $(u,g), (v,h)$ are \textit{equivalent} if and only if there exists a zigzag of 2-cells between them:
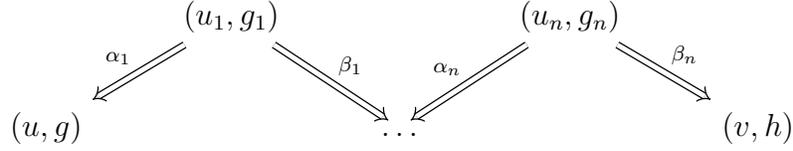
\[\begin{tikzcd}
	& {(u_1,g_1)} && {(u_n,g_n)} \\
	{(u,g)} && \dots && {(v,h)}
	\arrow["{\alpha_1}"', Rightarrow, from=1-2, to=2-1]
	\arrow["{\beta_1}", Rightarrow, from=1-2, to=2-3]
	\arrow["{\alpha_n}"', Rightarrow, from=1-4, to=2-3]
	\arrow["{\beta_n}", Rightarrow, from=1-4, to=2-5]
\end{tikzcd}\]

The identity on $a \in X$ is the equivalence class $[1_a, 1_a]$, while the composite $[v,h] \circ [u,g]$ is defined to be the equivalence class of corners obtained by filling the middle corner \textbf{with a choice} of an opcartesian square:
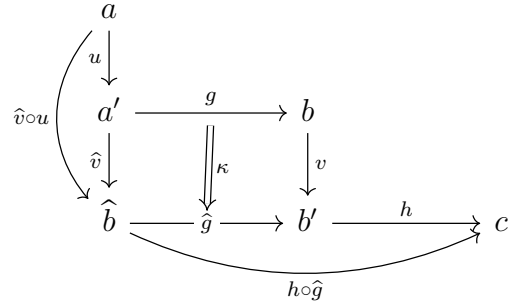
\[\begin{tikzcd}
	a \\
	{a'} && b \\
	{\widehat{b}} && {b'} && c
	\arrow["u"', from=1-1, to=2-1]
	\arrow[""{name=0, anchor=center, inner sep=0}, "g", from=2-1, to=2-3]
	\arrow["v", from=2-3, to=3-3]
	\arrow["h", from=3-3, to=3-5]
	\arrow["{\widehat{v}}"', from=2-1, to=3-1]
	\arrow[""{name=1, anchor=center, inner sep=0}, "{\widehat{g}}"{description}, from=3-1, to=3-3]
	\arrow["{h\circ \widehat{g}}"', curve={height=24pt}, from=3-1, to=3-5]
	\arrow["{\widehat{v}\circ u}"', curve={height=24pt}, from=1-1, to=3-1]
	\arrow["\kappa", shorten <=4pt, shorten >=4pt, Rightarrow, from=0, to=1]
\end{tikzcd}\]
Note that the composite is well defined on the equivalence classes. It is also independent of the choice of the square $\kappa$: if $\kappa'$ is another opcartesian square filling the corner, there is a unique square $\beta$ such that the following holds:
\[\begin{tikzcd}
	{a'} && b && {a'} &&& b \\
	{\widehat{b}'} && {b'} & {=} \\
	{\widehat{b}'} && {b'} && {\widehat{b}} &&& {b'}
	\arrow[""{name=0, anchor=center, inner sep=0}, "g", from=1-1, to=1-3]
	\arrow["v", from=1-3, to=2-3]
	\arrow["{v'}"', from=1-1, to=2-1]
	\arrow[""{name=1, anchor=center, inner sep=0}, "{g'}"{description}, from=2-1, to=2-3]
	\arrow[""{name=2, anchor=center, inner sep=0}, "{\widehat{g}}"', from=3-1, to=3-3]
	\arrow[Rightarrow, no head, from=2-3, to=3-3]
	\arrow[""{name=3, anchor=center, inner sep=0}, "g", from=1-5, to=1-8]
	\arrow["v", from=1-8, to=3-8]
	\arrow["{\widehat{v}}"', from=1-5, to=3-5]
	\arrow[""{name=4, anchor=center, inner sep=0}, "{\widehat{g}}"', from=3-5, to=3-8]
	\arrow["\theta"', from=2-1, to=3-1]
	\arrow["{\widehat{v}}"', curve={height=24pt}, from=1-1, to=3-1]
	\arrow["\kappa"', shorten <=9pt, shorten >=9pt, Rightarrow, from=3, to=4]
	\arrow["{\kappa'}", shorten <=4pt, shorten >=4pt, Rightarrow, from=0, to=1]
	\arrow["{\exists ! \beta}", shorten <=4pt, shorten >=4pt, Rightarrow, from=1, to=2]
\end{tikzcd}\]
This square exhibits now the equality between the compositions:
\[\begin{tikzcd}
	a \\
	{a'} \\
	{\widehat{b}'} && {b'} && c \\
	{\widehat{b}} && {b'} && c
	\arrow["{v'}"', from=2-1, to=3-1]
	\arrow["\theta"', from=3-1, to=4-1]
	\arrow["u"', from=1-1, to=2-1]
	\arrow["{\widehat{v}}"', curve={height=24pt}, from=2-1, to=4-1]
	\arrow[""{name=0, anchor=center, inner sep=0}, "{\widehat{g}}"', from=4-1, to=4-3]
	\arrow["h"', from=3-3, to=3-5]
	\arrow["h"', from=4-3, to=4-5]
	\arrow[Rightarrow, no head, from=3-5, to=4-5]
	\arrow[""{name=1, anchor=center, inner sep=0}, "{g'}", from=3-1, to=3-3]
	\arrow[Rightarrow, no head, from=3-3, to=4-3]
	\arrow["\beta", shorten <=4pt, shorten >=4pt, Rightarrow, from=1, to=0]
\end{tikzcd}\]
\end{con}

\begin{prop}\label{lemmacnrXiscat}
$\cnr{X}$ is a category.
\end{prop}
\begin{proof}
Let $[f,g]: a \to b$ be a morphism in $\cnr{X}$. To show that $[f,g] \circ [1_a,1_a] = [f,g]$, note that the horizontal identity square on $f$ is by definition opcartesian so we might as well use it for the composite (the composite is independent of the choice) and the result follows.
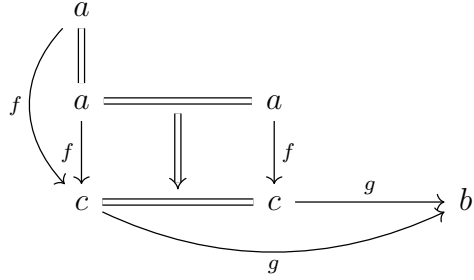
\[\begin{tikzcd}
	a \\
	a && a \\
	c && c && b
	\arrow[Rightarrow, no head, from=1-1, to=2-1]
	\arrow[""{name=0, anchor=center, inner sep=0}, Rightarrow, no head, from=2-1, to=2-3]
	\arrow["f", from=2-3, to=3-3]
	\arrow["g", from=3-3, to=3-5]
	\arrow["f"', from=2-1, to=3-1]
	\arrow[""{name=1, anchor=center, inner sep=0}, Rightarrow, no head, from=3-1, to=3-3]
	\arrow["f"', curve={height=24pt}, from=1-1, to=3-1]
	\arrow["g"', curve={height=24pt}, from=3-1, to=3-5]
	\arrow[shorten <=4pt, shorten >=4pt, Rightarrow, from=0, to=1]
\end{tikzcd}\]
Analogously $[1_b,1_b] \circ [f,g] = [f,g]$. Consider now a composable triple $[f,g]$, $[f',g']$, $[f'',g'']$ and fill it to form a single corner as depicted below:
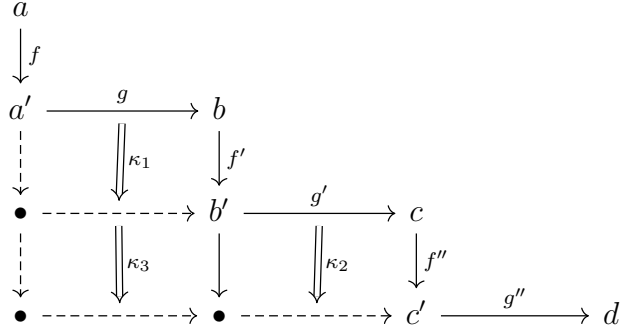
\[\begin{tikzcd}
	a \\
	{a'} && b \\
	\bullet && {b'} && c \\
	\bullet && \bullet && {c'} && d
	\arrow["f", from=1-1, to=2-1]
	\arrow[""{name=0, anchor=center, inner sep=0}, "g", from=2-1, to=2-3]
	\arrow["{f'}", from=2-3, to=3-3]
	\arrow[""{name=1, anchor=center, inner sep=0}, "{g'}", from=3-3, to=3-5]
	\arrow["{f''}", from=3-5, to=4-5]
	\arrow["{g''}", from=4-5, to=4-7]
	\arrow[dashed, from=2-1, to=3-1]
	\arrow[""{name=2, anchor=center, inner sep=0}, dashed, from=3-1, to=3-3]
	\arrow[from=3-3, to=4-3]
	\arrow[""{name=3, anchor=center, inner sep=0}, dashed, from=4-3, to=4-5]
	\arrow[dashed, from=3-1, to=4-1]
	\arrow[""{name=4, anchor=center, inner sep=0}, dashed, from=4-1, to=4-3]
	\arrow["{\kappa_1}", shorten <=4pt, shorten >=4pt, Rightarrow, from=0, to=2]
	\arrow["{\kappa_3}", shorten <=4pt, shorten >=4pt, Rightarrow, from=2, to=4]
	\arrow["{\kappa_2}", shorten <=4pt, shorten >=4pt, Rightarrow, from=1, to=3]
\end{tikzcd}\]
Denote this composite corner by $[f'',g''] \circ [f',g'] \circ [f,g]$ and call it the \textit{ternary composite}. Now to define the composition $[f'',g''] \circ [f',g']$, choose the square $\kappa_2$ as above. To define $([f'',g''] \circ [f',g']) \circ [f,g]$, choose the square $\kappa_3 \circ \kappa_1$ (as a vertical composite of opcartesian squares, it is opcartesian). We see that $([f'',g''] \circ [f',g']) \circ [f,g]$ is equal to the ternary composite. By an analogous argument (and using that opcartesian squares are closed under horizontal composites), $[f'',g''] \circ ([f',g'] \circ [f,g])$ also equals this ternary composite and thus composition is associative.

\end{proof}

If $F: X \to Y$ is any double functor between crossed double categories, there is an induced functor $\cnr{F}: \cnr{X} \to \cnr{Y}$ sending:
\begin{equation}\label{mapsto_cnrF}
\begin{tikzcd}
	a &&& Fa \\
	{a'} & b & \mapsto & {Fa'} & Fb
	\arrow["u"', from=1-1, to=2-1]
	\arrow["g", from=2-1, to=2-2]
	\arrow["Fu"', from=1-4, to=2-4]
	\arrow["Fg", from=2-4, to=2-5]
\end{tikzcd}
\end{equation}
This gives us a functor $\cnr{-}: \pscr \to \cat$. 

\begin{pozn}\label{pozn_coddiscrequivrel}
If $X$ is codomain-discrete, note that there is a 2-cell $\beta: (u,g) \Rightarrow (v,h)$ between corners if and only if $u=v$, $g=h$ and $\beta = 1_g$ is the identity square. Thus the category $\cnr{X}$ has corners as morphisms with no equivalence relation involved. This has also been observed in \cite[Corollary 5.4.7]{internalalg}.
\end{pozn}


\color{black}


\begin{pr}\label{prspans}
Let $\cc$ be a category with pullbacks and consider the double category $\text{PbSq}(\cc)^{v}$. The category $\cnr{\text{PbSq}(\cc)^{v}}$ has as objects the objects of $\cc$, while a morphism $a \to b$ is an equivalence class of corners (usually called spans) like this:
\[\begin{tikzcd}
	a \\
	{a'} & b
	\arrow[from=2-1, to=1-1]
	\arrow[from=2-1, to=2-2]
\end{tikzcd}\]
Note that there is a span isomorphism between two spans $(u,g), (v,h)$ if and only if there is a 2-cell between them if we regard them as corners.
\[\begin{tikzcd}
	a &&&&& a \\
	& {a'} && \iff && {a'} && b \\
	{a''} && b &&& {a'} && b
	\arrow["u"', from=2-2, to=1-1]
	\arrow["g", from=2-2, to=3-3]
	\arrow["v", from=3-1, to=1-1]
	\arrow["{\exists \cong}"', from=3-1, to=2-2]
	\arrow["h"', from=3-1, to=3-3]
	\arrow["u"', from=2-6, to=1-6]
	\arrow[""{name=0, anchor=center, inner sep=0}, "g", from=2-6, to=2-8]
	\arrow[Rightarrow, no head, from=2-8, to=3-8]
	\arrow[""{name=1, anchor=center, inner sep=0}, "g"', from=3-6, to=3-8]
	\arrow[from=3-6, to=2-6]
	\arrow["v", curve={height=-24pt}, from=3-6, to=1-6]
	\arrow["\exists", shorten <=4pt, shorten >=4pt, Rightarrow, from=0, to=1]
\end{tikzcd}\]
The composition of corners is defined using pullbacks. In other words, we have:
\[
\cnr{\text{PbSq}(\cc)^{v}} = \text{Span}(\cc).
\]
\end{pr}

\begin{pr}\label{prparc}
Let $\cc$ be a category with pullbacks and consider the double category $\text{MPbSq}(\cc)^{v}$. By a similar reasoning as above we obtain that the category of corners corresponding to this double category is isomorphic to the category $\text{Par}(\cc)$ of \textit{partial maps} in $\cc$, as defined in \cite[p. 246]{restrcats}.
\end{pr}

\begin{pr}\label{prcofun}
Let $\ce$ be a category with pullbacks and consider now the double category $\text{BOFib}(\ce)^{v}$. Morphisms in $\cnr{\text{BOFib}(\ce)^{v}}$ are equivalence classes of spans $(F,G)$ where $F$ is a bijection on objects and $G$ is a discrete opfibration as pictured below:
\[\begin{tikzcd}
	\ca \\
	{\ca'} && \cb
	\arrow["F", from=2-1, to=1-1]
	\arrow["G", from=2-1, to=2-3]
\end{tikzcd}\]
This category of corners is isomorphic to $\text{Cof}(\ce)$, the category of internal categories and \textit{cofunctors}, see for instance \cite[Theorem 18]{bryce_cof}.
\end{pr}

\begin{pr}
If $(L,R)$ is an algebraic weak factorization system on a category $\cc$ with pullbacks, consider the co-crossed double category $R\text{-}\mathbb{A}lg$ of $R$-algebras (Example \ref{pr_awfs_dblcat}). Its vertical dual $R\text{-}\mathbb{A}lg^{v}$ is thus crossed. Its category of corners construction now gives the \textit{category of weak maps} $\boldsymbol{Wk}_l(L,R)$ associated to the system $(L,R)$, see \cite[Section 3.4 and  Remark 13]{awfs2}.
\end{pr}

\begin{pr}
If $X$ is split crossed, our category of corners construction agrees with that of \cite[Corollary 5.4.5]{internalalg}, as is easily verified.

Recall from Example \ref{pr_crosseddblcats} the free symmetric strict monoidal category 2-monad $S$ and the crossed double category $\res{*}$.

The objects of $\cnr{\res{*}}$ are finite ordinals, while a morphism $m \to n$ is an equivalence class of corners consisting of a permutation followed by an order-preserving map:
\[\begin{tikzcd}
	m \\
	m & n
	\arrow["\phi"', from=1-1, to=2-1]
	\arrow["f"', from=2-1, to=2-2]
\end{tikzcd}\]
It can be proven that $\cnr{\res{*}} = \text{FinSet}$, the category of finite ordinals and all functions (see \cite[Theorem 6.3.1]{internalalg}).

If $B$ is the free braided strict monoidal category 2-monad, $\cnr{\res{*}} = \text{Vine}$, the category with objects being natural numbers and morphisms being \textit{vines}, that is, ``braids for which the strings can merge” (see \cite[Theorem 6.3.2]{internalalg}).
\end{pr}
\color{black}

\subsubsection{Some properties of $\cnr{X}$}


The following proposition captures the idea that ``a square in $X$ turns into a commutative square in $\cnr{X}$”:

\begin{prop}\label{prop_naturalityinsquares}
Let $X$ be a crossed double category. We have:
\[\begin{tikzcd}
	a && b &&& a && b \\
	&&& {\text{in } X} & \Rightarrow &&&& {\text{commutes in } \cnr{X}} \\
	c && d &&& c && d
	\arrow[""{name=0, anchor=center, inner sep=0}, "m", from=1-1, to=1-3]
	\arrow["e", from=1-3, to=3-3]
	\arrow["{e'}"', from=1-1, to=3-1]
	\arrow[""{name=1, anchor=center, inner sep=0}, "{m'}"', from=3-1, to=3-3]
	\arrow["{[e',1]}"', from=1-6, to=3-6]
	\arrow["{[1,m]}", from=1-6, to=1-8]
	\arrow["{[e,1]}", from=1-8, to=3-8]
	\arrow["{[1,m']}"', from=3-6, to=3-8]
	\arrow["\exists"', shorten <=9pt, shorten >=9pt, Rightarrow, from=0, to=1]
\end{tikzcd}\]
\end{prop}
\begin{proof}
Denote $[u,g] := [e,1] \circ [1.m]$ and denote by $\kappa$ the opcartesian square we used for this composition. From opcartesianness there is a unique square $\beta$ such that:
\[\begin{tikzcd}
	a && b && a &&& b \\
	{\widehat{b}'} && d & {=} \\
	c && d && c &&& d
	\arrow[""{name=0, anchor=center, inner sep=0}, "m", from=1-1, to=1-3]
	\arrow["e", from=1-3, to=2-3]
	\arrow["u"', from=1-1, to=2-1]
	\arrow[""{name=1, anchor=center, inner sep=0}, "g"{description}, from=2-1, to=2-3]
	\arrow[""{name=2, anchor=center, inner sep=0}, "{m'}"', from=3-1, to=3-3]
	\arrow[Rightarrow, no head, from=2-3, to=3-3]
	\arrow[""{name=3, anchor=center, inner sep=0}, "m", from=1-5, to=1-8]
	\arrow["e", from=1-8, to=3-8]
	\arrow["{e'}"', from=1-5, to=3-5]
	\arrow[""{name=4, anchor=center, inner sep=0}, "{m'}"', from=3-5, to=3-8]
	\arrow["\theta"', from=2-1, to=3-1]
	\arrow["{e'}"', curve={height=24pt}, from=1-1, to=3-1]
	\arrow["\alpha"', shorten <=9pt, shorten >=9pt, Rightarrow, from=3, to=4]
	\arrow["\kappa", shorten <=4pt, shorten >=4pt, Rightarrow, from=0, to=1]
	\arrow["{\exists ! \beta}", shorten <=4pt, shorten >=4pt, Rightarrow, from=1, to=2]
\end{tikzcd}\]
This square $\beta$ now exhibits the equality $[1,m'] \circ [e',1] = [e',m'] = [u,g] = [e,1] \circ [1,m]$.

\end{proof}

The additional assumption requiring that every square is opcartesian simplifies the description of $\cnr{X}$ for a crossed double category $X$:

\begin{lemma}\label{lemmaXpscr_everysquareopcart}
Let $X$ be a crossed double category in which every square is opcartesian. Then any 2-cell $\beta: [e,m] \Rightarrow [e',m']$ between corners is vertically invertible. In particular two corners in $\cnr{X}$ are equivalent if and only if there exists a single (vertically invertible) 2-cell between them.
\end{lemma}
\begin{proof}
Consider a square $\beta$ as pictured below. Since the vertical identity square $1_g$ on the morphism $g$ is opcartesian, there exists a unique square $\gamma$ such that:
\[\begin{tikzcd}
	a && b && a && b \\
	c && b & {=} \\
	a && b && a && b
	\arrow["u"', from=1-1, to=2-1]
	\arrow[""{name=0, anchor=center, inner sep=0}, "g", from=1-1, to=1-3]
	\arrow[Rightarrow, no head, from=1-3, to=2-3]
	\arrow[""{name=1, anchor=center, inner sep=0}, "h"{description}, from=2-1, to=2-3]
	\arrow["v"', from=2-1, to=3-1]
	\arrow[Rightarrow, no head, from=2-3, to=3-3]
	\arrow[""{name=2, anchor=center, inner sep=0}, "g"', from=3-1, to=3-3]
	\arrow["g", from=1-5, to=1-7]
	\arrow[Rightarrow, no head, from=1-7, to=3-7]
	\arrow[Rightarrow, no head, from=1-5, to=3-5]
	\arrow["g"', from=3-5, to=3-7]
	\arrow[curve={height=24pt}, Rightarrow, no head, from=1-1, to=3-1]
	\arrow["{\exists ! \gamma}", shorten <=4pt, shorten >=4pt, Rightarrow, from=1, to=2]
	\arrow["\beta", shorten <=4pt, shorten >=4pt, Rightarrow, from=0, to=1]
\end{tikzcd}\]
Thus $\gamma \circ \beta = 1_g$. Post-composing with $\beta$, we get:
\[
\beta \circ \gamma \circ \beta = 1_h \circ \beta
\]
Since $\beta$ is opcartesian, we get $\beta \circ \gamma = 1_h$ as well.
\end{proof}


\begin{nota}\label{notecnrX_verticalhorizontalcnrs}
Given a crossed double category $X$, denote by $\ce_X$ the class of corners in $\cnr{X}$ of form $[f,1_b]$ and $\cm_X$ the class of corners of form $[1_a, g]$. We will call these \textit{vertical} and \textit{horizontal} corners.
\end{nota}

\begin{prop}\label{THMcancprops_weakorthog}
Let $X$ be a crossed double category. Then the class $\ce_X$ has the right cancellation property. Both classes $\ce_X, \cm_X$ contain all identities and are closed under composition. We also have:
\[
\cnr{X} = \cm_X \circ \ce_X.
\]
Moreover, if every square is opcartesian in $X$, we have that $\ce_X$ is weakly orthogonal to $\cm_X$:
\[
\ce_X \boxslash \cm_X.
\]
\end{prop}
\begin{proof}
To show the cancellation property, assume $[s,t] \circ [e,1] = [u,1] \in \ce_X$. We then have a square as pictured below left:
\[\begin{tikzcd}
	a \\
	b &&&&& b \\
	{b'} && c &&& {a'} && c \\
	c && c &&& c && c
	\arrow["s", from=2-1, to=3-1]
	\arrow[""{name=0, anchor=center, inner sep=0}, "t", from=3-1, to=3-3]
	\arrow[""{name=1, anchor=center, inner sep=0}, Rightarrow, no head, from=4-1, to=4-3]
	\arrow[Rightarrow, no head, from=3-3, to=4-3]
	\arrow["e", from=1-1, to=2-1]
	\arrow["\theta"', from=3-1, to=4-1]
	\arrow["u"', curve={height=24pt}, from=1-1, to=4-1]
	\arrow["\theta"', from=3-6, to=4-6]
	\arrow[""{name=2, anchor=center, inner sep=0}, Rightarrow, no head, from=4-6, to=4-8]
	\arrow[""{name=3, anchor=center, inner sep=0}, "t", from=3-6, to=3-8]
	\arrow[Rightarrow, no head, from=3-8, to=4-8]
	\arrow["s", from=2-6, to=3-6]
	\arrow["{\theta s}"', curve={height=24pt}, from=2-6, to=4-6]
	\arrow["\beta", shorten <=4pt, shorten >=4pt, Rightarrow, from=0, to=1]
	\arrow["\beta", shorten <=4pt, shorten >=4pt, Rightarrow, from=3, to=2]
\end{tikzcd}\]
The same square $\beta$ now exhibits the equality $[s,t] = [\theta s, 1]$ (pictured above right) and thus $[s,t] \in \ce_X$. The fact that $\ce_X, \cm_X$ contain identities and are closed under composition, as well as the fact that $\cnr{X} = \cm_X \circ \ce_X$ are obvious.

Assume now that every square is opcartesian in $X$. Since $\cnr{X} = \cm_X \circ \ce_X$, to prove weak orthogonality it suffices to show that given two factorizations $[e,m] =  [e',m']$ of the same morphism, there exists a morphism of factorizations between them, i.e.:
\[\begin{tikzcd}
	a && {a'} && b \\
	a && {a''} && b
	\arrow["{[e,1]}", from=1-1, to=1-3]
	\arrow["{[1,m]}", from=1-3, to=1-5]
	\arrow["{[e',1]}"', from=2-1, to=2-3]
	\arrow["{[1,m']}"', from=2-3, to=2-5]
	\arrow[Rightarrow, no head, from=1-1, to=2-1]
	\arrow[Rightarrow, no head, from=1-5, to=2-5]
	\arrow[dashed, from=1-3, to=2-3]
\end{tikzcd}\]
Since $[e,m] = [e',m']$ and thanks to Lemma \ref{lemmaXpscr_everysquareopcart}, there exists a single invertible square like this:
\[\begin{tikzcd}
	a \\
	{a'} && b \\
	{a''} && b
	\arrow["{e'}"', curve={height=24pt}, from=1-1, to=3-1]
	\arrow["e"', from=1-1, to=2-1]
	\arrow[""{name=0, anchor=center, inner sep=0}, "m", from=2-1, to=2-3]
	\arrow["\theta"', from=2-1, to=3-1]
	\arrow[""{name=1, anchor=center, inner sep=0}, "{m'}"', from=3-1, to=3-3]
	\arrow[Rightarrow, no head, from=2-3, to=3-3]
	\arrow["\beta", shorten <=4pt, shorten >=4pt, Rightarrow, from=0, to=1]
\end{tikzcd}\]
It's now easy to verify that the corner $[\theta, 1]$ makes both squares in the above diagram commute.
\end{proof}

The classes $(\ce_X,\cm_X)$ give rise to an ordinary weak factorization system $(\widetilde{\ce_X},\widetilde{\cm_X})$ for which the first class is obtained by closing $\ce_X$ under codomain-retracts and the second is obtained from $\cm_X$ by closing it under domain retracts. This is a well-known result so we omit its proof.


\color{black}

\begin{pr}\label{pr_spansWFS}
Recall the example $\cnr{\pbsq{\cc}^{v}} = \spn{\cc}$. Since every square is opcartesian (a pullback), we obtain that the class $\ce_{\pbsq{\cc}^{v}}$ is weakly orthogonal to $\cm_{\pbsq{\cc}^{v}}$. Note that in this case, both classes are already closed under the required retracts. We obtain:

\begin{prop}\label{prop_wfsonspans}
The two canonical classes of morphisms in the category $\spn{\cc}$ form a weak factorization system.
\end{prop}
\end{pr}


\color{black}

\section{Factorization systems and double categories}\label{sekcefsdbl}

In this section we will be putting additional hypotheses on the crossed double category $X$ to ensure that the classes $(\ce_X,\cm_X)$ of morphisms in $\cnr{X}$ have more desirable properties (namely, form a strict or an orthogonal factorization system). This gives us the direction:
\[
\text{double categories } \rightsquigarrow \text{ factorization systems.}
\]
For the opposite direction, we introduce a construction that sends two classes $(\ce,\cm)$ of morphisms in a category $\cc$ to a certain double category $D_{\ce,\cm}$ of commutative squares.

In Subsection \ref{subsekce_sfs} we then show that the mappings $X \mapsto (\ce_X, \cm_X)$ and \\$(\ce,\cm) \mapsto D_{\ce,\cm}$ induce an equivalence between the categories of strict factorization systems and codomain-discrete double categories.

In Subsection \ref{subsekce_ofs} we prove analogous results for the categories of orthogonal factorization systems and  \textit{factorization double categories}: a symmetric variant of crossed double categories whose bottom-left corners satisfy a certain joint monicity property.



\subsection{Strict factorization systems}\label{subsekce_sfs}


In \cite{distrlaws} it has been shown that distributive laws in $\spanc$ can equivalently be described as strict factorization systems. Given a codomain-discrete double category $X$, the category of corners $\cnr{X}$ can be constructed using a distributive law in $\spanc(\cat)$\footnote{This is the original construction of $\cnr{X}$ for a codomain-discrete double category $X$ in \cite{internalalg}} - this gives a first hint that there is a relationship between double categories and factorization systems.

\begin{defi}
A \textit{strict factorization system} $(\ce,\cm)$ on a category $\cc$ consists of two wide sub-categories $\ce, \cm \subseteq \cc$ such that for every morphism $f \in \cc$ there exist unique $e \in \ce, m \in \cm$ such that: $f = m \circ e$.
\end{defi}


\begin{defi}\label{defi_SFS}
Denote by $\cs \cf \cs$ the category whose:
\begin{itemize}
\item objects are strict factorization systems $\ce \subseteq \cc \supseteq \cm$,
\item a morphism $(\ce \subseteq \cc \supseteq \cm) \to (\ce' \subseteq \cc' \supseteq \cm')$ is a functor $F: \cc \to \cc'$ satisfying:
\begin{align*}
F(\ce) &\subseteq \ce',\\
F(\cm) &\subseteq \cm'.
\end{align*}
\end{itemize}
\end{defi}

\begin{lemma}\label{lemma_coddiscrimpliessfs}
Let $X$ be codomain-discrete. Then the classes $(\ce_X,\cm_X)$ form a strict factorization system on $\cnr{X}$. The assignment $X \mapsto (\ce_X,\cm_X)$ induces a functor $\coddiscr \to \sfs$.
\end{lemma}
\begin{proof}
Recall from Remark \ref{pozn_coddiscrequivrel} that for such $X$, two morphisms $[u,g], [v,h]$ in $\cnr{X}$ are equal if and only if $u = v, g = h$. From this it follows that the factorization $[u,g] = [1,g] \circ [u,1]$ is unique.

If $H: X \to Y$ is a double functor, the induced functor $\cnr{H}: \cnr{X} \to \cnr{Y}$ (see \eqref{mapsto_cnrF}) satisfies $\cnr{H}(\ce_X) \subseteq \ce_Y$, $\cnr{H}(\cm_X) \subseteq \cm_Y$ and thus is a morphism in $\sfs$.
\end{proof}

Denote the above functor simply by $\cnr{-}: \coddiscr \to \sfs$.

\color{black}
\begin{pr}
Let $\ca, \cb$ be categories and $X_{\ca,\cb}$ the codomain-discrete double category from Example \ref{pr_atimesbdblcat}. The category of corners $\cnr{X_{\ca,\cb}}$ is isomorphic to just $\ca \times \cb$ and this category admits a strict factorization system $(\ce,\cm)$, where:
\begin{align*}
\ce &:= \{ (1_a, f) | a \in \ca, f \in \cb \},\\
\cm &:= \{ (g, 1_b) | g \in \ca, b \in \cb \}.
\end{align*}
\end{pr}
\color{black}

\begin{con}\label{conDcecm}
Let $(\ce,\cm)$ be two classes of morphisms in a category $\cc$, both closed under composition and containing all identities. We define a double category $D_{\ce,\cm}$ as follows:
\begin{itemize}
\item The objects are the objects of $\cc$,
\item the category of objects and vertical morphisms is $\ce$,
\item the category of objects and horizontal morphisms is $\cm$,
\item the squares are commutative squares in $\cc$.
\end{itemize}

If we have two classes $(\ce,\cm), (\ce',\cm')$ on categories $\cc, \cc'$ and $F: \cc \to \cc'$ a functor satisfying $F(\ce) \subseteq \ce'$ and $F(\cm) \subseteq \cm'$, there is an induced double functor $D_F : D_{\ce,\cm} \to D_{\ce',\cm'}$ defined in the obvious way.
\end{con}

\begin{lemma}\label{lemma_sfs_to_coddiscr}
Let $(\ce,\cm)$ be a strict factorization system on a category $\cc$. Then $D_{\ce,\cm}$ is codomain-discrete. The assignment $(\ce,\cm) \mapsto D_{\ce,\cm}$ induces a functor $\sfs \to \coddiscr$.
\end{lemma}
\begin{proof}
Every morphism in $\cc$ of form $e \circ m$ can be uniquely factored as $m' \circ e'$ with $e' \in \ce, m' \in \cm$. But this precisely means that $D_{\ce,\cm}$ is codomain-discrete.

In the above construction we have seen that $(\ce,\cm) \mapsto D_{\ce,\cm}$ is functorial, the rest is now obvious.
\end{proof}

\begin{theo}\label{thmekvivalence_sfs_coddiscr}
The functor $\cnr{-}: \coddiscr \to \sfs$ is the equivalence inverse to the functor $D: \sfs \to \coddiscr$ and so we have:
\[
\sfs \simeq \text{CodDiscr}.
\]
\end{theo}
\begin{proof}

We will show that there are natural isomorphisms $1 \cong D \circ \cnr{-}$ and\\$\cnr{-} \circ D \cong 1$:

To see that $1 \cong D \circ \cnr{-}$, let $X$ be a domain-discrete double category. First note that by Remark \ref{pozn_coddiscrequivrel}, the identity-on-objects functor $\ce \to \ce_X$ sending a morphism $e \mapsto (e,1)$ is an isomorphism. Similarly we have $\cm \cong \cm_X$ via the identity-on-objects functor $m \mapsto (1,m)$.

There is now a double functor $X \to D_{\ce_X, \cm_X}$ that is identity on objects and whose vertical morphism and horizontal morphism components are given by the functors described above. To see that it is well-defined on squares, we'd need to prove the direction ``$\Rightarrow$” in the picture below:
\begin{equation}\label{eq_iffsquares}
\begin{tikzcd}
	a && b &&& a && b \\
	&&& {\text{in } X} & \iff &&&& {\text{in }D_{\ce_X,\cm_X}} \\
	c && d &&& c && d
	\arrow[""{name=0, anchor=center, inner sep=0}, "m", from=1-1, to=1-3]
	\arrow["e", from=1-3, to=3-3]
	\arrow["{e'}"', from=1-1, to=3-1]
	\arrow[""{name=1, anchor=center, inner sep=0}, "{m'}"', from=3-1, to=3-3]
	\arrow["{(e',1)}"', from=1-6, to=3-6]
	\arrow[""{name=2, anchor=center, inner sep=0}, "{(m,1)}", from=1-6, to=1-8]
	\arrow["{(e,1)}", from=1-8, to=3-8]
	\arrow[""{name=3, anchor=center, inner sep=0}, "{(m',1)}"', from=3-6, to=3-8]
	\arrow["\exists"', shorten <=9pt, shorten >=9pt, Rightarrow, from=0, to=1]
	\arrow["\exists"', shorten <=9pt, shorten >=9pt, Rightarrow, from=2, to=3]
\end{tikzcd}
\end{equation}
But this direction follows from from Proposition \ref{prop_naturalityinsquares}. Since both double categories are flat, to show that this double functor is an isomorphism it suffices to show the direction ``$\Leftarrow$” in Diagram \eqref{eq_iffsquares}. Consider the square used for the composition $(e,1) \circ (m,1)$:
\[\begin{tikzcd}
	a && b \\
	\\
	\bullet && d
	\arrow[""{name=0, anchor=center, inner sep=0}, "m", from=1-1, to=1-3]
	\arrow["e", from=1-3, to=3-3]
	\arrow["{\widehat{e}}"', from=1-1, to=3-1]
	\arrow[""{name=1, anchor=center, inner sep=0}, "{\widehat{m}}"', from=3-1, to=3-3]
	\arrow[shorten <=9pt, shorten >=9pt, Rightarrow, from=0, to=1]
\end{tikzcd}\]
From the commutativity of the square on the above right we get $e' = \widehat{e}$, $m' = \widehat{m}$ and thus we obtain the left square in Diagram \eqref{eq_iffsquares}.

To see that $\cnr{-} \circ D \cong 1$, consider now a strict factorization system $(\ce,\cm)$ on $\cc$, we then have a functor $\cc \to \cnr{D_{\ce,\cm}}$ that is identity on objects and sends $f \mapsto (e,m)$, where $f = me$ is the unique factorization with $e \in \ce, m \in \cm$. This functor is an isomorphism and preserves the classes $\ce, \cm$. 
\end{proof}

\subsection{Orthogonal factorization systems}\label{subsekce_ofs}

\begin{defi}
An \textit{orthogonal factorization system} $(\ce,\cm)$ on a category $\cc$ consists of two wide sub-categories $\ce, \cm \subseteq \cc$ satisfying\footnote{Note that this definition is equivalent to the more standard one in which the orthogonality $\ce \perp \cm$ appears. See \cite[Theorem 3.7]{joyalfs}}:
\begin{itemize}
\item For every morphism $f \in \cc$ there exist $e \in \ce, m \in \cm$ such that $f = m \circ e$, and if $f = m'e'$ is a second factorization with $e' \in \ce, m' \in \cm$, there exists a unique morphism $\theta$ so that this commutes:
\[\begin{tikzcd}
	a && {a'} && b \\
	a && {a''} && b
	\arrow["e", from=1-1, to=1-3]
	\arrow["m", from=1-3, to=1-5]
	\arrow["{e'}"', from=2-1, to=2-3]
	\arrow["{m'}"', from=2-3, to=2-5]
	\arrow[Rightarrow, no head, from=1-1, to=2-1]
	\arrow[Rightarrow, no head, from=1-5, to=2-5]
	\arrow["{\exists ! \theta}", dashed, from=1-3, to=2-3]
\end{tikzcd}\]
\item we have that $\ce \cap \cm = \{ \text{isomorphisms in }\cc \}$.
\end{itemize}
\end{defi}


In the same way as in Definition \ref{defi_SFS} define the category $\ofs$ with objects orthogonal factorization systems and morphisms being functors preserving both classes.

To describe orthogonal factorization systems as certain double categories, we will introduce a more symmetric version of a crossed double category. Given a double category $X$, denote:
\begin{equation}\label{eq_Xhvezdnicka}
X^* := ((X^v)^h)^T.
\end{equation}
This is the double category obtained from $X$ by taking both the vertical and horizontal opposites as well as the transpose.

\begin{defi}
A square $\lambda$ in a double category $X$ will be called \textit{bicartesian} if it is opcartesian in both $X$ and $X^*$. In elementary terms, this means that given a square $\alpha$ with the same top-right corner as $\lambda$, there exist unique squares $\epsilon, \delta$ so that both the bottom left composite and the bottom right composite are equal to the square $\alpha$:
\[\begin{tikzcd}
	a && b \\
	c && d & a & b & a && a && b \\
	e && d & e & d & e && c && d
	\arrow[""{name=0, anchor=center, inner sep=0}, "g", from=1-1, to=1-3]
	\arrow["v", from=1-3, to=2-3]
	\arrow["u"', from=1-1, to=2-1]
	\arrow[""{name=1, anchor=center, inner sep=0}, "h"{description}, from=2-1, to=2-3]
	\arrow[Rightarrow, no head, from=2-3, to=3-3]
	\arrow["\theta"', from=2-1, to=3-1]
	\arrow[""{name=2, anchor=center, inner sep=0}, "k"', from=3-1, to=3-3]
	\arrow[""{name=3, anchor=center, inner sep=0}, Rightarrow, no head, from=2-6, to=2-8]
	\arrow["l"', from=2-6, to=3-6]
	\arrow[""{name=4, anchor=center, inner sep=0}, "\psi"', from=3-6, to=3-8]
	\arrow["u"', from=2-8, to=3-8]
	\arrow[""{name=5, anchor=center, inner sep=0}, "g", from=2-8, to=2-10]
	\arrow["v", from=2-10, to=3-10]
	\arrow[""{name=6, anchor=center, inner sep=0}, "h"', from=3-8, to=3-10]
	\arrow[""{name=7, anchor=center, inner sep=0}, "g", from=2-4, to=2-5]
	\arrow["l"', from=2-4, to=3-4]
	\arrow["v", from=2-5, to=3-5]
	\arrow[""{name=8, anchor=center, inner sep=0}, "k"', from=3-4, to=3-5]
	\arrow["{\exists ! \epsilon}", shorten <=4pt, shorten >=4pt, Rightarrow, from=1, to=2]
	\arrow["\lambda", shorten <=4pt, shorten >=4pt, Rightarrow, from=0, to=1]
	\arrow["\lambda", shorten <=4pt, shorten >=4pt, Rightarrow, from=5, to=6]
	\arrow["{\exists !  \delta}", shorten <=4pt, shorten >=4pt, Rightarrow, from=3, to=4]
	\arrow["\alpha"', shorten <=4pt, shorten >=4pt, Rightarrow, from=7, to=8]
\end{tikzcd}\]
\end{defi}

\begin{defi}
A double category $X$ is \textit{top-right bicrossed} if every top-right corner can be filled into a bicartesian square, and moreover bicartesian squares are closed under horizontal and vertical compositions and contain vertical and horizontal identities.
\end{defi}

In a top-right bicrossed double category $X$, the two conditions of being opcartesian in $X$ and in $X^*$ can be expressed as a single condition as follows:

\begin{lemma}\label{LEMMA_bicartdblopcart}
Let $\lambda$ be a bicartesian square in a top-right bicrossed double category $X$ and let $\alpha$ be any square with the same top-right corner. Then there exists a unique square $\beta$ such that this equation holds:
\[\begin{tikzcd}
	a && a && b && a &&&& b \\
	{\widehat{b}} && {\widehat{b}} && d & {=} \\
	d && {\widehat{b}} && d && d &&&& c
	\arrow[""{name=0, anchor=center, inner sep=0}, "f", from=1-3, to=1-5]
	\arrow["g", from=1-5, to=2-5]
	\arrow["{\pi_1}"', from=1-3, to=2-3]
	\arrow[""{name=1, anchor=center, inner sep=0}, "{\pi_2}"', from=2-3, to=2-5]
	\arrow[""{name=2, anchor=center, inner sep=0}, "{\widehat{v}}"', from=3-1, to=3-3]
	\arrow["{\widehat{u}}"', from=2-1, to=3-1]
	\arrow[""{name=3, anchor=center, inner sep=0}, "f", from=1-7, to=1-11]
	\arrow["g", from=1-11, to=3-11]
	\arrow["u"', from=1-7, to=3-7]
	\arrow[""{name=4, anchor=center, inner sep=0}, "v"', from=3-7, to=3-11]
	\arrow[""{name=5, anchor=center, inner sep=0}, Rightarrow, no head, from=2-1, to=2-3]
	\arrow[Rightarrow, no head, from=2-3, to=3-3]
	\arrow[Rightarrow, no head, from=1-1, to=1-3]
	\arrow[Rightarrow, no head, from=2-5, to=3-5]
	\arrow["{\pi_1}", from=1-1, to=2-1]
	\arrow["{\pi_2}"', from=3-3, to=3-5]
	\arrow["v"', curve={height=18pt}, from=3-1, to=3-5]
	\arrow["u"', curve={height=18pt}, from=1-1, to=3-1]
	\arrow["\lambda", shorten <=4pt, shorten >=4pt, Rightarrow, from=0, to=1]
	\arrow["\alpha", shorten <=9pt, shorten >=9pt, Rightarrow, from=3, to=4]
	\arrow["{\exists ! \beta}", shorten <=4pt, shorten >=4pt, Rightarrow, from=5, to=2]
\end{tikzcd}\]
\end{lemma}
\begin{proof}
From the definition of opcartesianness in $X^*$ there is a unique square $\gamma$ such that:
\begin{equation}\label{eq_gammalambdahoriz}
\begin{tikzcd}
	a && a && b && a && b \\
	c && {\widehat{b}} && d && c && d
	\arrow["f"', from=1-1, to=2-1]
	\arrow[""{name=0, anchor=center, inner sep=0}, "f"', from=1-7, to=2-7]
	\arrow["v", from=1-9, to=2-9]
	\arrow[""{name=1, anchor=center, inner sep=0}, "g", from=1-7, to=1-9]
	\arrow[""{name=2, anchor=center, inner sep=0}, "k"', from=2-7, to=2-9]
	\arrow[""{name=3, anchor=center, inner sep=0}, Rightarrow, no head, from=1-1, to=1-3]
	\arrow[""{name=4, anchor=center, inner sep=0}, "g", from=1-3, to=1-5]
	\arrow[""{name=5, anchor=center, inner sep=0}, "h"', from=2-3, to=2-5]
	\arrow[""{name=6, anchor=center, inner sep=0}, "v", from=1-5, to=2-5]
	\arrow["u"', from=1-3, to=2-3]
	\arrow[""{name=7, anchor=center, inner sep=0}, "\psi"', from=2-1, to=2-3]
	\arrow["\gamma", shorten <=4pt, shorten >=4pt, Rightarrow, from=3, to=7]
	\arrow["\lambda", shorten <=4pt, shorten >=4pt, Rightarrow, from=4, to=5]
	\arrow["\alpha", shorten <=4pt, shorten >=4pt, Rightarrow, from=1, to=2]
	\arrow["{=}"{description}, draw=none, from=6, to=0]
\end{tikzcd}
\end{equation}

Because the horizontal identity square on a vertical morphism $u$ is opcartesian in $X$ (because $X$ is top-right bicrossed), there exists a unique square $\gamma'$ such that:
\[\begin{tikzcd}
	a && a && a && b \\
	{\widehat{b}} && {\widehat{b}} & {=} \\
	{\widehat{b}} && d && c && d
	\arrow["u"', from=1-1, to=2-1]
	\arrow["\theta"', from=2-1, to=3-1]
	\arrow[""{name=0, anchor=center, inner sep=0}, Rightarrow, no head, from=2-1, to=2-3]
	\arrow[""{name=1, anchor=center, inner sep=0}, "\psi"', from=3-1, to=3-3]
	\arrow[Rightarrow, no head, from=1-1, to=1-3]
	\arrow["u", from=1-3, to=2-3]
	\arrow[Rightarrow, no head, from=2-3, to=3-3]
	\arrow[""{name=2, anchor=center, inner sep=0}, Rightarrow, no head, from=1-5, to=1-7]
	\arrow["u", from=1-7, to=3-7]
	\arrow["f"', from=1-5, to=3-5]
	\arrow[""{name=3, anchor=center, inner sep=0}, "\psi"', from=3-5, to=3-7]
	\arrow["\gamma", shorten <=9pt, shorten >=9pt, Rightarrow, from=2, to=3]
	\arrow["{\gamma'}", shorten <=4pt, shorten >=4pt, Rightarrow, from=0, to=1]
\end{tikzcd}\]
This gives us the \textbf{existence}. To prove the  \textbf{uniqueness}, let $\beta$ be a different square satisfying the equation in the Lemma. Then the composite $\beta \circ 1^\bullet_{\pi_1}$ (vertical composite of $\beta$ and the horizontal identity square on $\pi_1$) satisfies the equation \eqref{eq_gammalambdahoriz} (with $\beta \circ 1^\bullet_{\pi_1}$ in place of $\gamma$). Because $\lambda$ is opcartesian in $X^*$, this forces $\beta \circ  1^\bullet_{\pi_1} = \gamma = \gamma' \circ 1^\bullet_{\pi_1}$. Because $1^\bullet_{\pi_1}$ is opcartesian in $X$, this in turn forces $\beta = \gamma'$.
\end{proof}

\color{black}

\begin{pr}
In the double category $\sq{\cc}^{v}$, a square is bicartesian if and only if it's a pullback square. If $\cc$ has pullbacks, the double category $\sq{\cc}^{v}$ is top-right bicrossed.

Both $\text{MPbSq}(\cc)^{v}$ (for a category $\cc$ with pullbacks) and $\text{BOFib}(\ce)^{v}$ are top-right bicrossed with every square being bicartesian.
\end{pr}
\color{black}

We will now focus on double categories in which every top-right corner can be filled into a square and every square is bicartesian. Such double categories are automatically crossed so we may again use the category of corners construction. Notice that in this case, two corners $(e,m), (e',m')$ are equivalent if and only if there is a single 2-cell between them (Lemma \ref{lemmaXpscr_everysquareopcart}), and moreover every 2-cell has the following form (since the vertical identities on morphisms are bicartesian):
\begin{equation}\label{PIC_twocellforbicrosseddblcat}
\begin{tikzcd}
	a && a \\
	{a'} && {a'} && b \\
	{a''} && {a'} && b
	\arrow[Rightarrow, no head, from=2-5, to=3-5]
	\arrow["m", from=2-3, to=2-5]
	\arrow["e", from=1-3, to=2-3]
	\arrow[Rightarrow, no head, from=1-1, to=1-3]
	\arrow["e"', from=1-1, to=2-1]
	\arrow[""{name=0, anchor=center, inner sep=0}, Rightarrow, no head, from=2-1, to=2-3]
	\arrow["\theta"', from=2-1, to=3-1]
	\arrow[Rightarrow, no head, from=2-3, to=3-3]
	\arrow[""{name=1, anchor=center, inner sep=0}, "\psi"', from=3-1, to=3-3]
	\arrow["m"', from=3-3, to=3-5]
	\arrow["{e'}"', curve={height=24pt}, from=1-1, to=3-1]
	\arrow["{m'}"', curve={height=24pt}, from=3-1, to=3-5]
	\arrow["\beta", shorten <=4pt, shorten >=4pt, Rightarrow, from=0, to=1]
\end{tikzcd}
\end{equation}

The following somewhat technical notion is introduced in this paper to guarantee the uniqueness of factorizations up to a unique morphism:

\begin{defi}\label{defi_jointly_monic}
A top-left corner $(\pi_1, \pi_2)$ is said to be \textit{jointly monic} if, given squares $\kappa_1$, $\kappa_2$ pictured below:
\[\begin{tikzcd}
	{a'} && b && {a''} && {a'} && {a''} && {a'} \\
	\\
	a &&&& {a'} && {a'} && {a'} && {a'}
	\arrow["{\pi_2}", from=1-1, to=1-3]
	\arrow["{\pi_1}"', from=1-1, to=3-1]
	\arrow[""{name=0, anchor=center, inner sep=0}, "{\psi'}", from=1-9, to=1-11]
	\arrow["\theta"', from=1-5, to=3-5]
	\arrow[""{name=1, anchor=center, inner sep=0}, "\psi", from=1-5, to=1-7]
	\arrow[Rightarrow, no head, from=1-7, to=3-7]
	\arrow[""{name=2, anchor=center, inner sep=0}, Rightarrow, no head, from=3-5, to=3-7]
	\arrow["{\theta'}"', from=1-9, to=3-9]
	\arrow[""{name=3, anchor=center, inner sep=0}, Rightarrow, no head, from=3-9, to=3-11]
	\arrow[Rightarrow, no head, from=1-11, to=3-11]
	\arrow["{\kappa_2}", shorten <=9pt, shorten >=9pt, Rightarrow, from=0, to=3]
	\arrow["{\kappa_1}", shorten <=9pt, shorten >=9pt, Rightarrow, from=1, to=2]
\end{tikzcd}\]
We have the following implication:
\[
(\pi_1 \theta = \pi_1 \theta' \land \pi_2 \psi = \pi_2 \psi' ) \Rightarrow \theta = \theta', \psi = \psi'.
\]
\end{defi}

\color{black}
\begin{pr}\label{prsq_jointmonicity}
In $\sq{\cc}$ a top-right corner $(\pi_1,\pi_2)$ is jointly monic if and only if the pair of morphisms $(\pi_1,\pi_2)$ is jointly monic in the category $\cc$.
In $\pbsq{\cc}$ a top-right corner is jointly monic if and only if the pair is jointly monic in $\cc$ \textbf{with respect to all isomorphisms}.
\end{pr}

\begin{pr}\label{prmpbsq_pbsqsjointmonicity}
In the double category $\text{MPbSq}(\cc)$ every top-left corner is jointly monic as we now show: Let there be squares $\kappa_1, \kappa_2$ as in the definition, then $\theta = \psi$, $\theta' = \psi'$ and the equality $\pi_1 \theta = \pi_1 \theta'$ forces $\theta = \theta'$ because $\pi_1$ is a monomorphism.
\end{pr}

\begin{pr}\label{prbofib_pbsqsjointmonicity}
Every top-right corner is jointly monic in the double category $\text{BOFib}(\ce)$ as well: Let there be a top-left corner and squares as pictured below:
\[\begin{tikzcd}
	C && B && D && C && D && C \\
	\\
	A &&&& C && C && C && C
	\arrow["G", from=1-1, to=1-3]
	\arrow["F"', from=1-1, to=3-1]
	\arrow[""{name=0, anchor=center, inner sep=0}, "{P'}", from=1-9, to=1-11]
	\arrow["H"', from=1-5, to=3-5]
	\arrow[""{name=1, anchor=center, inner sep=0}, "P", from=1-5, to=1-7]
	\arrow[Rightarrow, no head, from=1-7, to=3-7]
	\arrow[""{name=2, anchor=center, inner sep=0}, Rightarrow, no head, from=3-5, to=3-7]
	\arrow["{H'}"', from=1-9, to=3-9]
	\arrow[""{name=3, anchor=center, inner sep=0}, Rightarrow, no head, from=3-9, to=3-11]
	\arrow[Rightarrow, no head, from=1-11, to=3-11]
	\arrow["{\kappa_2}", shorten <=9pt, shorten >=9pt, Rightarrow, from=0, to=3]
	\arrow["{\kappa_1}", shorten <=9pt, shorten >=9pt, Rightarrow, from=1, to=2]
\end{tikzcd}\]

Assume that $FH = FH'$, $GP = GP'$. We again get $H = P, H' = P'$. Now since $F$ is a bijection on objects, we have $H_0 = H'_0$ (the object parts of the functors agree). Because $G$ is a discrete opfibration, the square below is a pullback and we obtain $H_1 = H'_1$ as well:
\[\begin{tikzcd}
	{D_1} \\
	& {C_1} && {B_1} \\
	\\
	& {C_0} && {B_0}
	\arrow["s", from=2-4, to=4-4]
	\arrow["s"', from=2-2, to=4-2]
	\arrow["{G_1}", from=2-2, to=2-4]
	\arrow["{G_0}"', from=4-2, to=4-4]
	\arrow["\lrcorner"{anchor=center, pos=0.125}, draw=none, from=2-2, to=4-4]
	\arrow["{H_1'}", shift left=2, from=1-1, to=2-2]
	\arrow["{H_1}"', shift right=2, from=1-1, to=2-2]
	\arrow["{sH_0 = sH'_0}"', curve={height=24pt}, from=1-1, to=4-2]
	\arrow["{G_1H_1' = G_1H_1}", curve={height=-24pt}, from=1-1, to=2-4]
\end{tikzcd}\]
\end{pr}
\color{black}

\noindent Recall Notation \ref{notecnrX_verticalhorizontalcnrs}.

\begin{lemma}\label{lemma_isossubcxcm}
Let $X$ be a double category in which every top-right corner can be filled into a square and every square is bicartesian. Then:
\[
\ce_X \cap \cm_X \subseteq \{ \text{isomorphisms in } \cnr{X} \}.
\]
\end{lemma}
\begin{proof}
Let $[u,1] = [1,h]$ be a morphism in the intersection. There is then a 2-cell as follows (see the remark above diagram \eqref{PIC_twocellforbicrosseddblcat}):
\[\begin{tikzcd}
	a && a \\
	{a'} && {a'} && {a'} \\
	a && {a'} && {a'}
	\arrow["u", from=1-3, to=2-3]
	\arrow[Rightarrow, no head, from=2-3, to=2-5]
	\arrow["u"', from=1-1, to=2-1]
	\arrow["\theta"', from=2-1, to=3-1]
	\arrow[""{name=0, anchor=center, inner sep=0}, "h"', from=3-1, to=3-3]
	\arrow[Rightarrow, no head, from=2-3, to=3-3]
	\arrow[""{name=1, anchor=center, inner sep=0}, Rightarrow, no head, from=2-1, to=2-3]
	\arrow[Rightarrow, no head, from=3-3, to=3-5]
	\arrow[Rightarrow, no head, from=2-5, to=3-5]
	\arrow[Rightarrow, no head, from=1-1, to=1-3]
	\arrow[curve={height=18pt}, Rightarrow, no head, from=1-1, to=3-1]
	\arrow["h"', curve={height=18pt}, from=3-1, to=3-5]
	\arrow["\beta", shorten <=4pt, shorten >=4pt, Rightarrow, from=1, to=0]
\end{tikzcd}\]
Since every square is bicartesian, $\theta, h$ are isomorphisms, and so is $\theta^{-1} = u$. Hence $[u,1]$ is an isomorphism in $\cnr{X}$ with the inverse being $[u^{-1},1]$.

\end{proof}

\begin{lemma}
Let $X$ be a double category in which every top-right corner can be filled into a square and every square is bicartesian. An equivalence class of corners $[u,g]$ is invertible in $\cnr{X}$ if and only if both $u$ and $g$ are isomorphisms.
\end{lemma}
\begin{proof}
Let $[u,g]$ be an isomorphism in $\cnr{X}$ with inverse $[v,h]$ as pictured together with the inverse laws below:
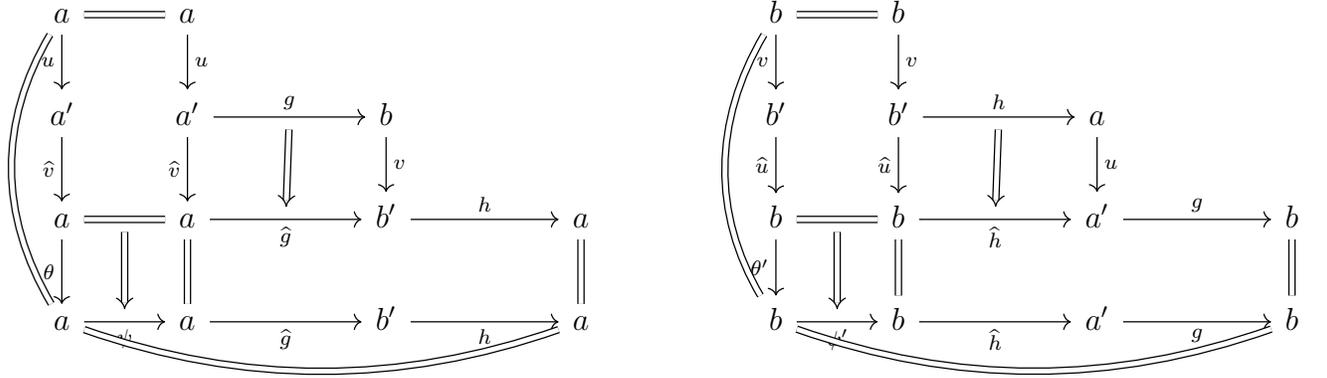
\begin{equation}\label{eq_ugvhinverse}
\begin{tikzcd}
	a & a &&&&&& b & b \\
	{a'} & {a'} && b &&&& {b'} & {b'} && a \\
	a & a && {b'} && a && b & b && {a'} && b \\
	a & a && {b'} && a && b & b && {a'} && b
	\arrow["u", from=1-2, to=2-2]
	\arrow[""{name=0, anchor=center, inner sep=0}, "g", from=2-2, to=2-4]
	\arrow["v", from=2-4, to=3-4]
	\arrow["h", from=3-4, to=3-6]
	\arrow["{\widehat{v}}"', from=2-2, to=3-2]
	\arrow[""{name=1, anchor=center, inner sep=0}, "{\widehat{g}}"', from=3-2, to=3-4]
	\arrow[""{name=2, anchor=center, inner sep=0}, Rightarrow, no head, from=3-1, to=3-2]
	\arrow[Rightarrow, no head, from=3-2, to=4-2]
	\arrow["\theta"', from=3-1, to=4-1]
	\arrow[""{name=3, anchor=center, inner sep=0}, "\psi"', from=4-1, to=4-2]
	\arrow["u"', from=1-1, to=2-1]
	\arrow["{\widehat{v}}"', from=2-1, to=3-1]
	\arrow[Rightarrow, no head, from=1-1, to=1-2]
	\arrow[Rightarrow, no head, from=3-6, to=4-6]
	\arrow["{\widehat{g}}"', from=4-2, to=4-4]
	\arrow["h"', from=4-4, to=4-6]
	\arrow[curve={height=24pt}, Rightarrow, no head, from=1-1, to=4-1]
	\arrow[curve={height=24pt}, Rightarrow, no head, from=4-1, to=4-6]
	\arrow["g", from=3-11, to=3-13]
	\arrow["v", from=1-9, to=2-9]
	\arrow[""{name=4, anchor=center, inner sep=0}, "h", from=2-9, to=2-11]
	\arrow["{\widehat{u}}"', from=2-9, to=3-9]
	\arrow[""{name=5, anchor=center, inner sep=0}, "{\widehat{h}}"', from=3-9, to=3-11]
	\arrow["u", from=2-11, to=3-11]
	\arrow["v"', from=1-8, to=2-8]
	\arrow["{\widehat{u}}"', from=2-8, to=3-8]
	\arrow[Rightarrow, no head, from=1-8, to=1-9]
	\arrow["{\widehat{h}}"', from=4-9, to=4-11]
	\arrow["g"', from=4-11, to=4-13]
	\arrow[Rightarrow, no head, from=3-13, to=4-13]
	\arrow[Rightarrow, no head, from=3-9, to=4-9]
	\arrow[""{name=6, anchor=center, inner sep=0}, Rightarrow, no head, from=3-8, to=3-9]
	\arrow["{\theta'}"', from=3-8, to=4-8]
	\arrow[""{name=7, anchor=center, inner sep=0}, "{\psi'}"', from=4-8, to=4-9]
	\arrow[curve={height=24pt}, Rightarrow, no head, from=1-8, to=4-8]
	\arrow[curve={height=24pt}, Rightarrow, no head, from=4-8, to=4-13]
	\arrow[shorten <=4pt, shorten >=4pt, Rightarrow, from=0, to=1]
	\arrow[shorten <=4pt, shorten >=4pt, Rightarrow, from=2, to=3]
	\arrow[shorten <=4pt, shorten >=4pt, Rightarrow, from=6, to=7]
	\arrow[shorten <=4pt, shorten >=4pt, Rightarrow, from=4, to=5]
\end{tikzcd}
\end{equation}

From the pictures below we obtain the following equalities:
\begin{align*}
[v,1] \circ [u,g] &= [1,\widehat{g}\psi], \\
[u,g] \circ [1,h] &= [\theta' \widehat{u},1].
\end{align*}
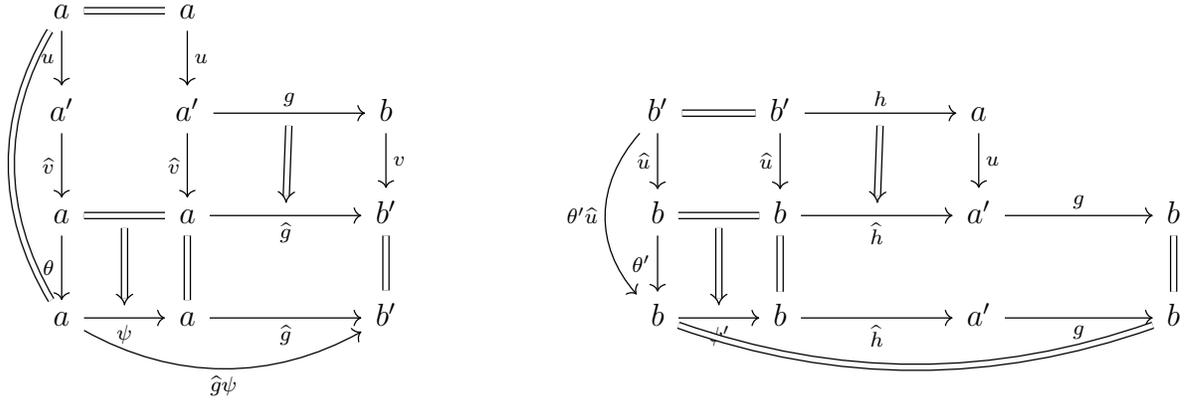
\[\begin{tikzcd}
	a & a \\
	{a'} & {a'} && b &&& {b'} & {b'} && a \\
	a & a && {b'} &&& b & b && {a'} && b \\
	a & a && {b'} &&& b & b && {a'} && b
	\arrow["u", from=1-2, to=2-2]
	\arrow[""{name=0, anchor=center, inner sep=0}, "g", from=2-2, to=2-4]
	\arrow["v", from=2-4, to=3-4]
	\arrow["{\widehat{v}}"', from=2-2, to=3-2]
	\arrow[""{name=1, anchor=center, inner sep=0}, "{\widehat{g}}"', from=3-2, to=3-4]
	\arrow[""{name=2, anchor=center, inner sep=0}, Rightarrow, no head, from=3-1, to=3-2]
	\arrow[Rightarrow, no head, from=3-2, to=4-2]
	\arrow["\theta"', from=3-1, to=4-1]
	\arrow[""{name=3, anchor=center, inner sep=0}, "\psi"', from=4-1, to=4-2]
	\arrow["u"', from=1-1, to=2-1]
	\arrow["{\widehat{v}}"', from=2-1, to=3-1]
	\arrow[Rightarrow, no head, from=1-1, to=1-2]
	\arrow["{\widehat{g}}"', from=4-2, to=4-4]
	\arrow[curve={height=24pt}, Rightarrow, no head, from=1-1, to=4-1]
	\arrow["g", from=3-10, to=3-12]
	\arrow[""{name=4, anchor=center, inner sep=0}, "h", from=2-8, to=2-10]
	\arrow["{\widehat{u}}"', from=2-8, to=3-8]
	\arrow[""{name=5, anchor=center, inner sep=0}, "{\widehat{h}}"', from=3-8, to=3-10]
	\arrow["u", from=2-10, to=3-10]
	\arrow["{\widehat{u}}"', from=2-7, to=3-7]
	\arrow["{\widehat{h}}"', from=4-8, to=4-10]
	\arrow["g"', from=4-10, to=4-12]
	\arrow[Rightarrow, no head, from=3-12, to=4-12]
	\arrow[Rightarrow, no head, from=3-8, to=4-8]
	\arrow[""{name=6, anchor=center, inner sep=0}, Rightarrow, no head, from=3-7, to=3-8]
	\arrow["{\theta'}"', from=3-7, to=4-7]
	\arrow[""{name=7, anchor=center, inner sep=0}, "{\psi'}"', from=4-7, to=4-8]
	\arrow[curve={height=24pt}, Rightarrow, no head, from=4-7, to=4-12]
	\arrow["{\widehat{g}\psi}"', curve={height=24pt}, from=4-1, to=4-4]
	\arrow[Rightarrow, no head, from=3-4, to=4-4]
	\arrow["{\theta'\widehat{u}}"', curve={height=24pt}, from=2-7, to=4-7]
	\arrow[Rightarrow, no head, from=2-7, to=2-8]
	\arrow[shorten <=4pt, shorten >=4pt, Rightarrow, from=0, to=1]
	\arrow[shorten <=4pt, shorten >=4pt, Rightarrow, from=2, to=3]
	\arrow[shorten <=4pt, shorten >=4pt, Rightarrow, from=6, to=7]
	\arrow[shorten <=4pt, shorten >=4pt, Rightarrow, from=4, to=5]
\end{tikzcd}\]

Now consider the following composite:
\[
[v \theta' \widehat{u},1] = [v,1] \circ [\theta' \widehat{u},1] = [v,1] \circ [u,g] \circ [1,h] = [1,\widehat{g}\psi] \circ [1,h] = [1,\widehat{g}\psi h]
\]

This composite belongs both in $\ce_X$ and $\cm_X$, so as in the proof of Lemma \ref{lemma_isossubcxcm} we obtain that $(\widehat{g}\psi) h$ is an isomorphism that we denote by $\Theta$. This implies that $\Theta^{-1}(\widehat{g}\psi) h = 1$ and so $h$ is a split monomorphism. Since $h (\widehat{g}\psi) = 1$ by Equation \eqref{eq_ugvhinverse}, $h$ is a split epimorphism. Thus $h$ is an isomorphism and by similar reasoning, $v$ is also an isomorphism. Hence $[v,h]$ is an isomorphism in $\cnr{X}$ with the inverse being given by $[v^{-1},1] \circ [1,h^{-1}]$.
\end{proof}

Note that for $X = \pbsq{\cc}^{v}$ the above lemma gives the usual folklore characterization of isomorphisms in the category $\spn{\cc}$ of spans.

\begin{lemma}\label{lemma_isossupcxcm}
Let $X$ be a double category in which every top-right corner can be filled into a square and every square is bicartesian. Assume in addition that $X$ is invariant. Then:
\[
\ce_X \cap \cm_X \supseteq \{ \text{isomorphisms in } \cnr{X} \}.
\]
\end{lemma}
\begin{proof}

We will show that when $X$ is horizontally invariant, we have:
\[
[u,g] \in \cnr{X}\text{, }g\text{ is an isomorphism }\Rightarrow [u,g] \in \ce_X.
\]

Let $[u,g]$ be such a corner. From horizontal invariance we get the square (pictured below), that exhibits the equality $[u,g] = [\theta u, 1]$:
\[\begin{tikzcd}
	a \\
	{a'} & b \\
	b & b
	\arrow["u"', from=1-1, to=2-1]
	\arrow[""{name=0, anchor=center, inner sep=0}, "g", from=2-1, to=2-2]
	\arrow[Rightarrow, no head, from=2-2, to=3-2]
	\arrow[""{name=1, anchor=center, inner sep=0}, Rightarrow, no head, from=3-1, to=3-2]
	\arrow["\theta"', from=2-1, to=3-1]
	\arrow["{\theta u}"', curve={height=24pt}, from=1-1, to=3-1]
	\arrow[shorten <=4pt, shorten >=4pt, Rightarrow, from=0, to=1]
\end{tikzcd}\]
Thus, $[u,g] \in \ce_X$. Dually, if $X$ is vertically invariant, we have:
\[
[u,g] \in \cnr{X}\text{, }u\text{ is an isomorphism }\Rightarrow [u,g] \in \cm_X.
\]

Now if $[u,g]$ is an isomorphism, by previous lemma both $u,g$ are isomorphisms, and by the above implications $[u,g]$ belongs to both $\ce_X$ and $\cm_X$.
\end{proof}

\begin{lemma}\label{lemma_orthog}
Let $X$ be a double category in which every top-right corner can be filled into a square and every square is bicartesian. Assume further that every top-left corner in $X^{v}$ is jointly monic. Then the $(\ce_X,\cm_X)$-factorization of a morphism in $\cnr{X}$ is unique up to a unique morphism.
\end{lemma}
\begin{proof}
Assume that $[e,m] = [e',m']$ are two $(\ce_X,\cm_X)$-factorizations of a morphism in $\cnr{X}$. We wish to show that there is a unique morphism between them:
\begin{equation}\label{eqmorfismusfaktorizaci}
\begin{tikzcd}
	a && {a'} && b \\
	a && {a''} && b
	\arrow["{[e,1]}", from=1-1, to=1-3]
	\arrow["{[1,m]}", from=1-3, to=1-5]
	\arrow["{[e',1]}"', from=2-1, to=2-3]
	\arrow["{[1,m']}"', from=2-3, to=2-5]
	\arrow[Rightarrow, no head, from=1-1, to=2-1]
	\arrow[Rightarrow, no head, from=1-5, to=2-5]
	\arrow[dashed, from=1-3, to=2-3]
\end{tikzcd}
\end{equation}

As in the proof of Theorem \ref{THMcancprops_weakorthog}, one such morphism is given by the corner $[\theta,1]$, where $\theta$ is the domain of the 2-cell square between $(e,m)$ and $(e',m')$:
\[\begin{tikzcd}
	a && a \\
	{a'} && {a'} && b \\
	{a''} && {a'} && b
	\arrow["e", from=1-3, to=2-3]
	\arrow["m", from=2-3, to=2-5]
	\arrow["e"', from=1-1, to=2-1]
	\arrow["\theta"', from=2-1, to=3-1]
	\arrow[""{name=0, anchor=center, inner sep=0}, "\psi"', from=3-1, to=3-3]
	\arrow[Rightarrow, no head, from=2-3, to=3-3]
	\arrow[""{name=1, anchor=center, inner sep=0}, Rightarrow, no head, from=2-1, to=2-3]
	\arrow["m"', from=3-3, to=3-5]
	\arrow[Rightarrow, no head, from=2-5, to=3-5]
	\arrow[Rightarrow, no head, from=1-1, to=1-3]
	\arrow["{e'}"', curve={height=18pt}, from=1-1, to=3-1]
	\arrow["{m'}"', curve={height=18pt}, from=3-1, to=3-5]
	\arrow["\beta", shorten <=4pt, shorten >=4pt, Rightarrow, from=1, to=0]
\end{tikzcd}\]

Assume that there is a different morphism $[s,t]: a' \to a''$ making both squares in \eqref{eqmorfismusfaktorizaci} commute. The commutativity of these two squares gives the following 2-cells:
\[\begin{tikzcd}
	a && a \\
	{a'} && {a'} &&&& {a'} && {a'} \\
	x && x && {a''} && x && x && {a''} && b \\
	{a''} && x && {a''} && {a'} && x && {a''} && b
	\arrow["s", from=2-3, to=3-3]
	\arrow["t", from=3-3, to=3-5]
	\arrow["s"', from=2-1, to=3-1]
	\arrow["{\theta'}"', from=3-1, to=4-1]
	\arrow[""{name=0, anchor=center, inner sep=0}, "{\psi'}"', from=4-1, to=4-3]
	\arrow[Rightarrow, no head, from=3-3, to=4-3]
	\arrow[""{name=1, anchor=center, inner sep=0}, Rightarrow, no head, from=3-1, to=3-3]
	\arrow["t"', from=4-3, to=4-5]
	\arrow[Rightarrow, no head, from=3-5, to=4-5]
	\arrow[curve={height=18pt}, Rightarrow, no head, from=4-1, to=4-5]
	\arrow["e", from=1-1, to=2-1]
	\arrow["e", from=1-3, to=2-3]
	\arrow[Rightarrow, no head, from=1-1, to=1-3]
	\arrow["{m'}", from=3-11, to=3-13]
	\arrow["s"', from=2-7, to=3-7]
	\arrow["s", from=2-9, to=3-9]
	\arrow[""{name=2, anchor=center, inner sep=0}, Rightarrow, no head, from=3-7, to=3-9]
	\arrow[Rightarrow, no head, from=3-9, to=4-9]
	\arrow["t", from=3-9, to=3-11]
	\arrow["t"', from=4-9, to=4-11]
	\arrow[Rightarrow, no head, from=3-13, to=4-13]
	\arrow["{m'}"', from=4-11, to=4-13]
	\arrow[""{name=3, anchor=center, inner sep=0}, "{\widetilde{\psi}}"', from=4-7, to=4-9]
	\arrow["{\widetilde{\theta}}"', from=3-7, to=4-7]
	\arrow[Rightarrow, no head, from=2-7, to=2-9]
	\arrow[curve={height=18pt}, Rightarrow, no head, from=2-7, to=4-7]
	\arrow["{e'}"', curve={height=24pt}, from=1-1, to=4-1]
	\arrow["m"', curve={height=24pt}, from=4-7, to=4-13]
	\arrow["{\beta'}", shorten <=4pt, shorten >=4pt, Rightarrow, from=1, to=0]
	\arrow["{\widetilde{\beta}}", shorten <=4pt, shorten >=4pt, Rightarrow, from=2, to=3]
\end{tikzcd}\]

Assume now we had the following:
\begin{align}\label{eqthetathetatilda}
\begin{split}
\theta \widetilde{\theta} &= \theta', \\
\widetilde{\psi} \psi &= \psi'.
\end{split}
\end{align}

The square $\beta'$ would then exhibit the equality $[s,t] = [\theta, 1]$:
\[\begin{tikzcd}
	{a'} && {a'} \\
	x && x && {a''} \\
	{a''} && x && {a''}
	\arrow[Rightarrow, no head, from=1-1, to=1-3]
	\arrow[Rightarrow, no head, from=2-5, to=3-5]
	\arrow[""{name=0, anchor=center, inner sep=0}, Rightarrow, no head, from=2-1, to=2-3]
	\arrow[Rightarrow, no head, from=2-3, to=3-3]
	\arrow["s"', from=1-1, to=2-1]
	\arrow["s", from=1-3, to=2-3]
	\arrow["t", from=2-3, to=2-5]
	\arrow["t"', from=3-3, to=3-5]
	\arrow["{\theta'}"', from=2-1, to=3-1]
	\arrow[""{name=1, anchor=center, inner sep=0}, "{\psi'}"', from=3-1, to=3-3]
	\arrow["\theta"', curve={height=24pt}, from=1-1, to=3-1]
	\arrow[curve={height=24pt}, Rightarrow, no head, from=3-1, to=3-5]
	\arrow["{\beta'}", shorten <=4pt, shorten >=4pt, Rightarrow, from=0, to=1]
\end{tikzcd}\]

Because the corner $(se, m't)$ is jointly monic in $X^{v}$, to show \eqref{eqthetathetatilda} it suffices to show:
\begin{align*}
\theta \widetilde{\theta} se &= \theta' se, \\
m't \widetilde{\psi} \psi &= m't \psi'.
\end{align*}
The first equality holds because:
\[
\theta \widetilde{\theta} s e = \theta e = e' = \theta' s e,
\]
while the second equality holds because:
\[
m't\widetilde{\psi} \psi = m \psi = m' = m' t \psi'.
\]
\end{proof}

We therefore propose the following terminology:

\begin{defi}
By an (orthogonal) \textit{factorization double category} we mean a double category $X$ with the following properties:
\begin{itemize}
\item $X$ is invariant,
\item every top-right corner in $X$ can be filled into a square and every square is bicartesian,
\item every top-left corner in $X^{v}$ is jointly monic.
\end{itemize}
Denote by $\factdbl$ the full subcategory of $\text{Dbl}$ consisting of factorization double categories.
\end{defi}

\begin{pozn}
Any factorization double category is automatically flat: Given two squares $\alpha, \lambda$ with the same boundary, by Lemma \ref{LEMMA_bicartdblopcart} there exists a unique square $\beta$ as follows:
\[\begin{tikzcd}
	a && a && b && a &&&& b \\
	d && d && d & {=} \\
	d && d && d && d &&&& c
	\arrow[""{name=0, anchor=center, inner sep=0}, "f", from=1-3, to=1-5]
	\arrow["g", from=1-5, to=2-5]
	\arrow["u"', from=1-3, to=2-3]
	\arrow[""{name=1, anchor=center, inner sep=0}, "v"', from=2-3, to=2-5]
	\arrow[""{name=2, anchor=center, inner sep=0}, "\psi"', from=3-1, to=3-3]
	\arrow["\theta"', from=2-1, to=3-1]
	\arrow[""{name=3, anchor=center, inner sep=0}, "f", from=1-7, to=1-11]
	\arrow["g", from=1-11, to=3-11]
	\arrow["u"', from=1-7, to=3-7]
	\arrow[""{name=4, anchor=center, inner sep=0}, "v"', from=3-7, to=3-11]
	\arrow[""{name=5, anchor=center, inner sep=0}, Rightarrow, no head, from=2-1, to=2-3]
	\arrow[Rightarrow, no head, from=2-3, to=3-3]
	\arrow[Rightarrow, no head, from=1-1, to=1-3]
	\arrow[Rightarrow, no head, from=2-5, to=3-5]
	\arrow["u"', from=1-1, to=2-1]
	\arrow["v"', from=3-3, to=3-5]
	\arrow["v"', curve={height=18pt}, from=3-1, to=3-5]
	\arrow["u"', curve={height=18pt}, from=1-1, to=3-1]
	\arrow["\lambda", shorten <=4pt, shorten >=4pt, Rightarrow, from=0, to=1]
	\arrow["\alpha", shorten <=9pt, shorten >=9pt, Rightarrow, from=3, to=4]
	\arrow["{\exists ! \beta}", shorten <=4pt, shorten >=4pt, Rightarrow, from=5, to=2]
\end{tikzcd}\]
Now in $X^{v}$ the corner $(u,v)$ is jointly monic and we have $\theta u = 1_d u$, $v \psi = v 1_d$. Thus $\theta = 1_d$ and $v = 1_d$. Invariance now forces $\gamma = \Box_d$, the identity square on the morphism $1_d$ and so $\alpha = \lambda$.
\end{pozn}

Combining Lemmas \ref{lemma_isossubcxcm}, \ref{lemma_isossupcxcm}, \ref{lemma_orthog} we obtain:

\begin{prop}\label{THM_dbl_to_ofs}
Let $X$ be a factorization double category. Then the classes $(\ce_X,\cm_X)$ of vertical and horizontal corners form an orthogonal factorization system on the category $\cnr{X}$.
\end{prop}


\color{black}
\begin{pr}[Partial maps]

Let $\cc$ be a category with pullbacks and consider the double category $\text{MPbSq}(\cc)^{v}$. It is obviously flat and invariant with very square bicartesian. We have seen that corners in its vertical dual are jointly monic in Example \ref{prmpbsq_pbsqsjointmonicity}. $\text{MPbSq}(\cc)^{v}$ is thus a factorization double category.

Combined with the description of the category of corners from Example \ref{prparc} and Proposition \ref{THM_dbl_to_ofs} we obtain that the category $\text{Par}(\cc)$ of objects and partial maps admits an orthogonal factorization system given by vertical corners followed by horizontal ones:
\[\begin{tikzcd}
	A &&&& A \\
	{A'} && {A'} && A && B
	\arrow[Rightarrow, no head, from=2-1, to=2-3]
	\arrow["\iota", hook, from=2-1, to=1-1]
	\arrow["f"', from=2-5, to=2-7]
	\arrow[Rightarrow, no head, from=1-5, to=2-5]
\end{tikzcd}\]
In \cite{restrcats} these are called the \textit{domains} and \textit{total maps} in $\cc$.
\end{pr}

\begin{pr}[Categories and cofunctors]
If $\ce$ is a category with pullbacks, the double category $\text{BOFib}(\ce)^{v}$ is a factorization double category. In the category\\ $\cnr{\text{BOFib}(\ce)^{v}} = \text{Cof}(\ce)$ every morphism can be factored as (the opposite of) a bijection on objects functor followed by a discrete opfibration (as mentioned in \cite[Theorem 18]{bryce_cof}). By the results in this section, these classes form an orthogonal factorization system on $\text{Cof}(\ce)$.
\end{pr}

\begin{pr}\label{pr_xp_fib}
Given a fibration $P: \ce \to \cb$, there is an associated sub-double category $X_P \subseteq \sq{\ce}$ whose vertical morphisms are $P$-vertical morphisms (those that are sent to isomorphisms by $P$), horizontal morphisms are cartesian lifts of morphisms in $\cb$, and squares are commutative squares in $\ce$.
\begin{prop}
$X_P$ is a factorization double category.
\end{prop}
\begin{proof}
Invariance is straightforward. To show joint monicity, assume we're given the data as in Definition \ref{defi_jointly_monic}. Note then that from the existence of squares $\kappa_1, \kappa_2$ in $(X^P)^v$ it follows that $P\theta = (P\psi)^{-1}$, and from $\theta \pi_1 = \theta' \pi_1$ we have $P\theta = P\theta'$. Since $\pi_2$ is a cartesian lift, the following picture forces $\psi \circ \theta' = 1$ and thus $\psi = \psi'$:
\[\begin{tikzcd}
	& {a'} && P && {Pa'} \\
	{a'} && b & \mapsto & {Pa'} && Pb
	\arrow["{\pi_2}", from=1-2, to=2-3]
	\arrow["{\pi_2}"', from=2-1, to=2-3]
	\arrow["{P\pi_2}", from=1-6, to=2-7]
	\arrow["{P\pi_2}"', from=2-5, to=2-7]
	\arrow[Rightarrow, no head, from=2-5, to=1-6]
	\arrow["{\psi \circ \theta'}", from=2-1, to=1-2]
\end{tikzcd}\]
Given a top-right corner $\lambda, u$, the bicartesian filler square is given by the cartesian lift of the pair $(P\lambda, b')$ and the unique canonical comparison morphism:
\[\begin{tikzcd}
	a && b && a & b & {b'} \\
	{a'} && {b'} && Pa && Pb
	\arrow[""{name=0, anchor=center, inner sep=0}, "\lambda", from=1-1, to=1-3]
	\arrow["u", from=1-3, to=2-3]
	\arrow[""{name=1, anchor=center, inner sep=0}, "{\lambda_{P\lambda, b}}"', from=2-1, to=2-3]
	\arrow["{\exists !}"', from=1-1, to=2-1]
	\arrow["\lambda", from=1-5, to=1-6]
	\arrow["u", from=1-6, to=1-7]
	\arrow[""{name=2, anchor=center, inner sep=0}, "P\lambda"', from=2-5, to=2-7]
	\arrow[shorten <=4pt, shorten >=4pt, Rightarrow, from=0, to=1]
	\arrow["\mapsto"{marking}, draw=none, from=1-6, to=2]
\end{tikzcd}\]
\end{proof}
The category of corners $\cnr{X_P}$ is isomorphic to the category $\ce$ via the functor sending an equivalence class $[u,\lambda]$ to the composite $\lambda \circ u$.

From the results in this section we obtain that the category $\ce$ admits an orthogonal factorization system given by the class of $P$-vertical morphisms followed by cartesian morphisms. Factorization systems associated to fibrations are a special case of \textit{simple reflective factorization systems} associated to prefibrations and have been studied in \cite{factorizationOFSrosicky}.
\end{pr}


\begin{nonpr}[Spans]

Let $\cc$ be a category with pullbacks and consider the double category $\text{PbSq}(\cc)^{v}$ of (opposite) pullback squares. This is not a factorization double category because not every top-left corner in $\text{PbSq}(\cc)$ is jointly monic (see Example \ref{prsq_jointmonicity}). We can not thus use Theorem \ref{THM_dbl_to_ofs} to obtain an orthogonal factorization system on $\cnr{\text{PbSq}(\cc)^{v}} = \text{Span}(\cc)$ and in fact, the two canonical classes of spans in $\spn{\cc}$ do not form one\footnote{They do form a weak factorization system as we've seen in Example \ref{pr_spansWFS}}.

To see this, let $\cc = \set$, denote by $\text{sw}: 2 \to 2$ the non-identity automorphism of the two-element set. Note that the class $[sw,1]: 2 \to 2$ in $\text{Span}$ is not the identity morphism. Consider now the span $(!,!): * \to *$:
\[\begin{tikzcd}
	{*} \\
	2 && {*}
	\arrow["{!}", from=2-1, to=1-1]
	\arrow["{!}", from=2-1, to=2-3]
\end{tikzcd}\]

\noindent As both $[sw,1]$ and $[1,1]$ in the place of the dotted line make the diagram below commute, we obtain that the factorization is not unique up to a unique isomorphism and thus the classes are not orthogonal.
\[\begin{tikzcd}
	{*} && 2 && {*} \\
	{*} && 2 && {*}
	\arrow[dashed, from=1-3, to=2-3]
	\arrow["{[!,1]}"', from=1-3, to=1-1]
	\arrow["{[!,1]}", from=2-3, to=2-1]
	\arrow["{[1,!]}"', from=2-3, to=2-5]
	\arrow["{[1,!]}", from=1-3, to=1-5]
	\arrow[Rightarrow, no head, from=1-1, to=2-1]
	\arrow[Rightarrow, no head, from=1-5, to=2-5]
\end{tikzcd}\]

\end{nonpr}
\color{black}
\bigskip

Recall now the assignment $(\ce,\cm) \mapsto D_{\ce,\cm}$ from Construction \ref{conDcecm}. We have:

\begin{prop}\label{lemma_ofs_to_dbl}
Let $(\ce,\cm)$ be an orthogonal factorization system on a category $\cc$. Then $D_{\ce,\cm}$ is a factorization double category. The assignment $(\ce,\cm) \to D_{\ce,\cm}$ induces a functor $\ofs \to \factdbl$.
\end{prop}
\begin{proof}
Let's verify each point:
\begin{itemize}


\item \textbf{Invariance}: We show the horizontal invariance, the vertical invariance is done similarly. Because the classes $\ce,\cm$ are closed under composition, given $u \in \ce$ and two isomorphisms $\theta, \psi$, the composite $\psi^{-1} u \theta$ gives the unique square with the given boundary:
\[\begin{tikzcd}
	a && b \\
	\\
	c && d
	\arrow[""{name=0, anchor=center, inner sep=0}, "{\theta \cong}", from=1-1, to=1-3]
	\arrow["u", from=1-3, to=3-3]
	\arrow[""{name=1, anchor=center, inner sep=0}, "\psi\cong"', from=3-1, to=3-3]
	\arrow[dashed, from=1-1, to=3-1]
	\arrow[shorten <=9pt, shorten >=9pt, Rightarrow, from=0, to=1]
\end{tikzcd}\]

\item \textbf{Filling corners into squares, every square bicartesian}: Consider a top-right corner as pictured below:
\[\begin{tikzcd}
	a && b \\
	\\
	{a'} && c
	\arrow[""{name=0, anchor=center, inner sep=0}, "{m'}", from=1-1, to=1-3]
	\arrow["{e'}", from=1-3, to=3-3]
	\arrow["e"', dashed, from=1-1, to=3-1]
	\arrow[""{name=1, anchor=center, inner sep=0}, "m"', dashed, from=3-1, to=3-3]
	\arrow[shorten <=9pt, shorten >=9pt, Rightarrow, from=0, to=1]
\end{tikzcd}\]
The filler square is given by the $(\ce,\cm)$-factorization of the morphism $e'm'$ in $\cc$. Next, let there be a square as pictured below left:
\[\begin{tikzcd}
	a && b && a && b \\
	{a'} && c && {a''} && c
	\arrow[""{name=0, anchor=center, inner sep=0}, "g", from=1-1, to=1-3]
	\arrow["u", from=1-3, to=2-3]
	\arrow[""{name=1, anchor=center, inner sep=0}, "g", from=1-5, to=1-7]
	\arrow["u", from=1-7, to=2-7]
	\arrow["e"', from=1-1, to=2-1]
	\arrow[""{name=2, anchor=center, inner sep=0}, "m"', from=2-1, to=2-3]
	\arrow["{e'}"', from=1-5, to=2-5]
	\arrow[""{name=3, anchor=center, inner sep=0}, "{m'}"', from=2-5, to=2-7]
	\arrow[shorten <=4pt, shorten >=4pt, Rightarrow, from=0, to=2]
	\arrow[shorten <=4pt, shorten >=4pt, Rightarrow, from=1, to=3]
\end{tikzcd}\]
We wish to show that it is bicartesian. Assume there is another square (pictured above right) with the same top-right corner.

Because $me = m'e'$ are two factorizations of the same morphism, there is a unique isomorphism $\theta \in \ce \cap \cm$ between them (pictured below left). It then gives a comparison square between the first square and the second square, as pictured below right:
\[\begin{tikzcd}
	& a &&& a && b && a && b \\
	{a''} && {a'} && {a'} && c & {=} \\
	& c &&& {a''} && c && {a''} && c
	\arrow[Rightarrow, no head, from=2-7, to=3-7]
	\arrow["e"', from=1-5, to=2-5]
	\arrow[""{name=0, anchor=center, inner sep=0}, "m"{description}, from=2-5, to=2-7]
	\arrow[""{name=1, anchor=center, inner sep=0}, "{m'}"', from=3-5, to=3-7]
	\arrow[""{name=2, anchor=center, inner sep=0}, "g", from=1-5, to=1-7]
	\arrow["u", from=1-7, to=2-7]
	\arrow["e", from=1-2, to=2-3]
	\arrow["m", from=2-3, to=3-2]
	\arrow["{e'}"', from=1-2, to=2-1]
	\arrow["{m'}"', from=2-1, to=3-2]
	\arrow["{\cong \theta}"{description}, from=2-3, to=2-1]
	\arrow[""{name=3, anchor=center, inner sep=0}, "g", from=1-9, to=1-11]
	\arrow["u", from=1-11, to=3-11]
	\arrow["{e'}"', from=1-9, to=3-9]
	\arrow[""{name=4, anchor=center, inner sep=0}, "{m'}"', from=3-9, to=3-11]
	\arrow["\theta"', from=2-5, to=3-5]
	\arrow[shorten <=4pt, shorten >=4pt, Rightarrow, from=2, to=0]
	\arrow[shorten <=9pt, shorten >=9pt, Rightarrow, from=3, to=4]
	\arrow[shorten <=4pt, shorten >=4pt, Rightarrow, from=0, to=1]
\end{tikzcd}\]
This gives the \textbf{existence}. To prove the \textbf{uniqueness}, assume that there is a different comparison square. Its vertical domain map then gives the morphism of factorizations $(e,m),(e',m')$ and is thus forced to be equal to $\theta$. Thus the square is opcartesian in $D_{\ce,\cm}$. The proof that it is also opcartesian in ${(D_{\ce,\cm})}^*$ is done the same way.


\item \textbf{Joint monicity}: Let $(\pi_1, \pi_2)$, $\kappa_1$, $\kappa_2$ be the data in $D_{\ce,\cm}$ as pictured below:
\[\begin{tikzcd}
	a &&&& {a'} && {a'} && {a'} && {a'} \\
	{a'} && b && {a''} && {a'} && {a''} && {a'}
	\arrow["{\pi_2}", from=2-1, to=2-3]
	\arrow["{\pi_1}", from=1-1, to=2-1]
	\arrow[""{name=0, anchor=center, inner sep=0}, "{\psi'}"', from=2-9, to=2-11]
	\arrow["\theta"', from=1-5, to=2-5]
	\arrow[""{name=1, anchor=center, inner sep=0}, "\psi"', from=2-5, to=2-7]
	\arrow[Rightarrow, no head, from=2-7, to=1-7]
	\arrow[""{name=2, anchor=center, inner sep=0}, Rightarrow, no head, from=1-5, to=1-7]
	\arrow["{\theta'}"', from=1-9, to=2-9]
	\arrow[""{name=3, anchor=center, inner sep=0}, Rightarrow, no head, from=1-9, to=1-11]
	\arrow[Rightarrow, no head, from=2-11, to=1-11]
	\arrow["{\kappa_2}"', shorten <=4pt, shorten >=4pt, Rightarrow, from=3, to=0]
	\arrow["{\kappa_1}"', shorten <=4pt, shorten >=4pt, Rightarrow, from=2, to=1]
\end{tikzcd}\]

Assume $\theta \pi_1 = \theta' \pi_1 $ and $\pi_2 \psi = \pi_2 \psi'$. We have the following:
\begin{align*}
\psi' \theta \pi_1 &= \pi_1,\\
\pi_2 &= \pi_2 \psi' \theta.
\end{align*}

In other words, $\psi' \theta$ is an endomorphism of the factorization:
\[\begin{tikzcd}
	a && {a'} && b \\
	a && {a'} && b
	\arrow["{\pi_1}", from=1-1, to=1-3]
	\arrow["{\pi_2}", from=1-3, to=1-5]
	\arrow["{\pi_1}"', from=2-1, to=2-3]
	\arrow["{\pi_2}"', from=2-3, to=2-5]
	\arrow[Rightarrow, no head, from=1-1, to=2-1]
	\arrow[Rightarrow, no head, from=1-5, to=2-5]
	\arrow["{\psi'\circ \theta}", from=1-3, to=2-3]
\end{tikzcd}\]
From orthogonality we get $\psi' \theta = 1$. Since $\psi \circ \theta$ is an $(\ce,\cm)$-factorization of the identity, both of $\psi$ and $\theta$ are isomorphisms together with $\psi' \theta = 1$ we get:
\[
\psi = \theta^{-1} = \psi'.
\]
Applying inverses, we get $\theta = \theta'$.
\end{itemize}
Thus $D_{\ce,\cm}$ is a factorization double category. The functoriality is straightforward.
\end{proof}

Analogous to Theorem \ref{thmekvivalence_sfs_coddiscr}, we have:
\begin{theo}\label{thm_equiv_ofs_factdbl}
The functor $\cnr{-}: \factdbl \to \ofs$ is the equivalence inverse to the functor $D: \ofs \to \factdbl$ and so we have:
\[
\ofs \simeq \factdbl.
\]
\end{theo}
\begin{proof}
We will again show that there are natural isomorphisms $1 \cong D \circ \cnr{-}$ and $\cnr{-} \circ D \cong 1$:

To see that $1 \cong D \circ \cnr{-}$, let $X$ be a factorization double category.
Consider again the identity-on-objects functor $\ce \to \ce_X$ that sends the morphism $e \mapsto [e,1]$, we would like to show that it's an isomorphism. It is clearly full. To see faithfulness, assume we have $[e,1] = [e',1]$, then there is a 2-cell like this:
\[\begin{tikzcd}
	a \\
	{a'} & {a'} \\
	{a'} & {a'}
	\arrow[""{name=0, anchor=center, inner sep=0}, Rightarrow, no head, from=2-1, to=2-2]
	\arrow["\theta"', from=2-1, to=3-1]
	\arrow[""{name=1, anchor=center, inner sep=0}, Rightarrow, no head, from=3-1, to=3-2]
	\arrow[Rightarrow, no head, from=2-2, to=3-2]
	\arrow["e"', from=1-1, to=2-1]
	\arrow["{e'}"', curve={height=24pt}, from=1-1, to=3-1]
	\arrow[shorten <=4pt, shorten >=4pt, Rightarrow, from=0, to=1]
\end{tikzcd}\]
Since the double category $X$ is invariant, the square is forced to be the identity and thus $e = e'$. Analogously there is an isomorphism $\cm \cong \cm_X$.  Define a double functor $X \to D_{\ce_X, \cm_X}$ so that it is identity on objects, on vertical morphisms sends $e \mapsto [e,1]$ and on horizontal morphisms sends $m \mapsto [1,m]$. Because both double categoroies are flat, to show that it is an isomorphism it's enough to prove the following:
\[\begin{tikzcd}
	a && b &&& a && b \\
	&&& {\text{in } X} & \iff &&&& {\text{in }D_{\ce_X,\cm_X}} \\
	c && d &&& c && d
	\arrow[""{name=0, anchor=center, inner sep=0}, "m", from=1-1, to=1-3]
	\arrow["e", from=1-3, to=3-3]
	\arrow["{e'}"', from=1-1, to=3-1]
	\arrow[""{name=1, anchor=center, inner sep=0}, "{m'}"', from=3-1, to=3-3]
	\arrow["{[e',1]}"', from=1-6, to=3-6]
	\arrow[""{name=2, anchor=center, inner sep=0}, "{[m,1]}", from=1-6, to=1-8]
	\arrow["{[e,1]}", from=1-8, to=3-8]
	\arrow[""{name=3, anchor=center, inner sep=0}, "{[m',1]}"', from=3-6, to=3-8]
	\arrow["\exists"', shorten <=9pt, shorten >=9pt, Rightarrow, from=0, to=1]
	\arrow["\exists"', shorten <=9pt, shorten >=9pt, Rightarrow, from=2, to=3]
\end{tikzcd}\]

The direction ``$\Rightarrow$” follows already from Proposition \ref{prop_naturalityinsquares}. For the direction ``$\Leftarrow$” assume the right square commutes. If we denote $[e,1] \circ [1,m] = [u,g]$, we obtain the square required above right as the following composite, where the upper square is the square used for the composition of the corners, and the lower square exists from the equality $[u,g] = [e',m']$:
\[\begin{tikzcd}
	a && b \\
	c && d \\
	c && d
	\arrow[""{name=0, anchor=center, inner sep=0}, "m", from=1-1, to=1-3]
	\arrow["e", from=1-3, to=2-3]
	\arrow["u"', from=1-1, to=2-1]
	\arrow[""{name=1, anchor=center, inner sep=0}, "g"{description}, from=2-1, to=2-3]
	\arrow["\theta"', from=2-1, to=3-1]
	\arrow[""{name=2, anchor=center, inner sep=0}, "{m'}"', from=3-1, to=3-3]
	\arrow[Rightarrow, no head, from=2-3, to=3-3]
	\arrow["{e'}"', curve={height=24pt}, from=1-1, to=3-1]
	\arrow[shorten <=4pt, shorten >=4pt, Rightarrow, from=0, to=1]
	\arrow[shorten <=4pt, shorten >=4pt, Rightarrow, from=1, to=2]
\end{tikzcd}\]

To see that $\cnr{-} \circ D \cong 1$, let $(\ce,\cm)$ be an orthogonal factorization system on a category $\cc$ and define an identity-on-objects functor $F: \cnr{D_{\ce,\cm}} \to \cc$ so that it sends the class of corners $[e,m]$ to $m \circ e$. This is well defined because if $[e,m] = [e',m']$, there is a 2-cell between $(e,m), (e',m')$ in $D_{\ce,\cm}$:
\[\begin{tikzcd}
	a && a \\
	{a'} && {a'} && b \\
	{a''} && {a'} && b
	\arrow[Rightarrow, no head, from=2-5, to=3-5]
	\arrow["m", from=2-3, to=2-5]
	\arrow["e", from=1-3, to=2-3]
	\arrow[Rightarrow, no head, from=1-1, to=1-3]
	\arrow["e"', from=1-1, to=2-1]
	\arrow[""{name=0, anchor=center, inner sep=0}, Rightarrow, no head, from=2-1, to=2-3]
	\arrow["\theta"', from=2-1, to=3-1]
	\arrow[Rightarrow, no head, from=2-3, to=3-3]
	\arrow[""{name=1, anchor=center, inner sep=0}, "\psi"', from=3-1, to=3-3]
	\arrow["m"', from=3-3, to=3-5]
	\arrow["{e'}"', curve={height=24pt}, from=1-1, to=3-1]
	\arrow["{m'}"', curve={height=24pt}, from=3-1, to=3-5]
	\arrow[shorten <=4pt, shorten >=4pt, Rightarrow, from=0, to=1]
\end{tikzcd}\]
But since squares in this double category are commutative squares in $\cc$, we get: $me = m \psi \theta e = m'e'$. Functoriality of $F$ is straightforward, faithfulness follows from the fact that $(\ce,\cm)$-factorizations are unique up to a unique isomorphism, and fullness follows because every morphism in $\cc$ has an $(\ce,\cm)$-factorization. Thus $F$ is an isomorphism.
\end{proof}

\section{Codescent objects and double categories}\label{sekcecodobjdbl}

The purpose of this section is to put the previous ones into the broader perspective of 2-category theory. In Subsection \ref{subsekce_introcodobj} we give a brief exposition of a 2-categorical colimit called the \textit{codescent object}. In Subsection \ref{subsekce_codobjdbl} we study codescent objects of double categories: if $X$ is crossed, the codescent object is given by the category of corners $\cnr{X}$ - this was in fact the original reason for introducing the category of corners of crossed double categories in \cite{internalalg}. As a colimit, the category $\cnr{X}$ enjoys a universal property: it is the universal category equipped with functors $(F: X_0 \to \cnr{X}, \xi: h(X) \to \cnr{X})$ satisfying a certain \textit{naturality condition} (Proposition \ref{prop_cnriscodobj}).

In Subsection \ref{subsekce_laxmors} we use the results to obtain an explicit description of various lax morphism classifiers. Section \ref{subsekce_sfs} then gives us a conceptual reason for why they admit strict factorization systems.

\subsection{Review of codescent objects}\label{subsekce_introcodobj}
Denote by $\Delta$ the category of finite ordinals and order-preserving maps and by $\Delta_2$ its full subcategory spanned by the ordinals $[0], [1], [2]$. Let $W: \Delta_2 \to \cat$ be the 2-functor that regards each ordinal as a category.

\begin{defi}
Let $\ck$ be a 2-category. By (\textit{strict, reflexive}) \textit{coherence data} in $\ck$ we mean a 2-functor $X: \Delta^{op}_2 \to \ck$. By the \textit{lax codescent object} of $X$ we mean the $W$-weighted colimit of $X$.
\end{defi}

\noindent In elementary terms, a $W$-weighted cocone for coherence data $X$ is a pair $(F,\xi)$ of a 1-cell $F: X_0 \to Y$ and a 2-cell $\xi: Fd_1 \Rightarrow Fd_0$ satisfying the following equations:
\[\begin{tikzcd}
	& {X_1} & {X_0} &&&& {X_1} & {X_0} \\
	{X_2} & {X_1} && Y & {=} & {X_2} && {X_0} & Y \\
	& {X_1} & {X_0} &&&& {X_1} & {X_0}
	\arrow["{d_0}"', from=2-1, to=3-2]
	\arrow["{d_1}"{description}, from=2-1, to=2-2]
	\arrow["{d_0}"', from=3-2, to=3-3]
	\arrow["{d_0}"{description}, from=2-2, to=3-3]
	\arrow["{d_1}"{description}, from=2-2, to=1-3]
	\arrow["F"', from=3-3, to=2-4]
	\arrow["F", from=1-3, to=2-4]
	\arrow["{d_2}", from=2-1, to=1-2]
	\arrow["{d_1}", from=1-2, to=1-3]
	\arrow["{d_0}"', from=2-6, to=3-7]
	\arrow["{d_2}", from=2-6, to=1-7]
	\arrow["{d_0}"', from=1-7, to=2-8]
	\arrow["{d_1}", from=3-7, to=2-8]
	\arrow["{d_0}"', from=3-7, to=3-8]
	\arrow["F"', from=3-8, to=2-9]
	\arrow["F"{description}, from=2-8, to=2-9]
	\arrow["{d_1}", from=1-7, to=1-8]
	\arrow["F", from=1-8, to=2-9]
	\arrow["\xi"', shorten <=6pt, shorten >=6pt, Rightarrow, from=1-3, to=3-3]
	\arrow["\xi", shorten <=2pt, shorten >=2pt, Rightarrow, from=1-8, to=2-8]
	\arrow["\xi", shorten <=2pt, shorten >=2pt, Rightarrow, from=2-8, to=3-8]
\end{tikzcd}\]
\[\begin{tikzcd}
	&&&& {X_0} \\
	{X_0} && {X_1} &&&& Y & {=} & {1_F} \\
	&&&& {X_0}
	\arrow["s"', from=2-1, to=2-3]
	\arrow["{d_0}"', from=2-3, to=3-5]
	\arrow["{d_1}", from=2-3, to=1-5]
	\arrow["F"', from=3-5, to=2-7]
	\arrow["F", from=1-5, to=2-7]
	\arrow["\xi", shorten <=6pt, shorten >=6pt, Rightarrow, from=1-5, to=3-5]
	\arrow[curve={height=-24pt}, Rightarrow, no head, from=2-1, to=1-5]
	\arrow[curve={height=24pt}, Rightarrow, no head, from=2-1, to=3-5]
\end{tikzcd}\]


A $W$-weighted cocone $(F,\xi)$ with apex $Y$ is then the lax codescent object for $X$ if given any other cocone $(G,\psi)$ with apex $Z$, there exists a unique map $\theta: Y \to Z$ such that\footnote{This is the 1-dimensional universal property. We omit mentioning the corresponding 2-dimensional universal property here since for the case we'll be interested in ($\ck = \cat$) it follows automatically from the 1-dimensional one. For the full definition see \cite[2.2]{johnphd}.}:
\begin{align}
\theta F &= G,\\
\theta \xi &= \psi.
\end{align}

\subsection{Codescent objects and double categories}\label{subsekce_codobjdbl}

\begin{lemma}\label{propcoconeistwofuns}
Let $X$ be a double category regarded as a diagram $X: \Delta^{op} \to \cat$. There is a natural bijection between codescent cocones for $X$ and pairs of functors $(F: X_0 \to \cy, \xi: h(X) \to \cy)$ (recall $h(X)$ is the category of objects and horizontal morphisms in $X$) that agree on objects and satisfy the following \textit{naturality condition}:
\[\begin{tikzcd}
	a & b \\
	c & d & {\text{ exists in }X} & \Rightarrow & {Fv \circ \xi(g) = \xi(h) \circ Fu \text{ in } \cy}
	\arrow[""{name=0, anchor=center, inner sep=0}, "g", from=1-1, to=1-2]
	\arrow["v", from=1-2, to=2-2]
	\arrow["u"', from=1-1, to=2-1]
	\arrow[""{name=1, anchor=center, inner sep=0}, "h"', from=2-1, to=2-2]
	\arrow["{\exists \alpha}", shorten <=4pt, shorten >=4pt, Rightarrow, from=0, to=1]
\end{tikzcd}\]
\end{lemma}
\begin{proof}
A cocone $(F,\xi)$ for $X$ consists of a functor $F:X_0 \to \cy$ and a natural transformation $\xi: Fd_1 \Rightarrow Fd_0: X_1 \to \cy$. The cocone axioms mean precisely that given a composable pair of morphisms $(h,g)$ in $h(X)$ and an object $a \in X$, we have:
\begin{align*}
\xi_g \circ \xi_h &= \xi_{g \circ h},\\
\xi_{s(a)} &= 1_{Fa}.
\end{align*}
In other words, $\xi$ induces a functor $h(X) \to \cy$ that sends an object $a \in h(X)$ to $Fa$ and a morphism $g \in h(X)$ to $\xi_g$. The naturality condition above is precisely the condition that $\xi$ is a natural transformation.
\end{proof}


Similarly, the cocone $(F,\xi)$ is the codescent object if and only if the corresponding pair of functors is initial in the sense that given a different pair of functors $(G,\psi)$ satisfying the naturality condition, there is a unique map commuting with the functors:
\[\begin{tikzcd}
	&& \cy \\
	{X_0} &&&& {h(X)} \\
	&& \cz
	\arrow["F", from=2-1, to=1-3]
	\arrow["\xi"', from=2-5, to=1-3]
	\arrow["G"', from=2-1, to=3-3]
	\arrow["\psi", from=2-5, to=3-3]
	\arrow[dashed, from=1-3, to=3-3]
\end{tikzcd}\]

From now on, we will not distinguish between codescent cocones $(F,\xi)$ and pairs of functors satisfying the conditions above.

\begin{prop}[Invariance under transposition]\label{propcodinvariantwrttransp}
Let $X$ be a double category. We have: $\cod{X} \cong \cod{X^T}$.
\end{prop}
\begin{proof}
Because $\cat$ admits cotensors with an arrow, it's enough to show just the one-dimensional universal property\footnote{See \cite[Page 306]{elemobs}}, i.e. show that there's a natural bijection between the sets of $W$-weighted codescent cocones:
\[
\text{Cocone}(X,\cc) \cong \text{Cocone}(X^T, \cc)
\]
The bijection is given by $(F,\xi) \mapsto (\xi,F)$.
\end{proof}

\begin{prop}\label{prop_cnriscodobj}
Let $X$ be a crossed double category. Then the pair of functors $(F: X_0 \to \cnr{X},\xi: h(X) \to \cnr{X})$ sending $u \mapsto [u,1]$, $g \mapsto [1,g]$ is the codescent object of $X$.
\end{prop}
\begin{proof}
The naturality condition has already been verified in Proposition \ref{prop_naturalityinsquares}. Let now $(G: X_0 \to Y, \psi: h(X) \to Y)$ be another cocone. We have to show that there is a unique functor $\theta: \cnr{X} \to Y$ commuting with the functors:
\[\begin{tikzcd}
	&& {\cnr{X}} \\
	{X_0} &&&& {h(X)} \\
	&& Y
	\arrow["F", from=2-1, to=1-3]
	\arrow["\xi"', from=2-5, to=1-3]
	\arrow["G"', from=2-1, to=3-3]
	\arrow["\psi", from=2-5, to=3-3]
	\arrow[dashed, from=1-3, to=3-3]
\end{tikzcd}\]
The above commutativity forces $\theta(a) = Fa = \xi(a)$ for an object $a \in \cnr{X}$, and on morphisms $\theta([u,g]) = \psi(g) \circ Gu$. It is routine to verify that this mapping is well defined and a functor.
\end{proof}

\begin{pozn}
Note that if $X$ is a general double category, the category of corners can be generalized in terms of generators and relations as follows. The codescent object $\cod{X}$ has objects the objects of $X$, while a morphism is an equivalence class of paths $[f_1, \dots, f_n]$, with $f_i$ being either a vertical or a horizontal morphism of $X$. The equivalence relation on morphisms is then generated by the following:
\begin{align*}
[f_1, f_2] &= [f_2 \circ f_1] \text{ if both }f_1, f_2\text{ are vertical or horizontal},\\
[1_a] &= []_a \text{ if }1_a\text{ is a vertical or a horizontal identity morphism in }X,\\
[g,v] &= [u,h] \text{ if there is a square }\alpha \text{ as pictured below:}
\end{align*}
\[\begin{tikzcd}
	a & b \\
	c & d
	\arrow[""{name=0, anchor=center, inner sep=0}, "g", from=1-1, to=1-2]
	\arrow["v", from=1-2, to=2-2]
	\arrow["u"', from=1-1, to=2-1]
	\arrow[""{name=1, anchor=center, inner sep=0}, "h"', from=2-1, to=2-2]
	\arrow["\alpha", shorten <=4pt, shorten >=4pt, Rightarrow, from=0, to=1]
\end{tikzcd}\]
To give an example, let $G,H$ be groups. We can define a double category $X$ with one object, vertical morphisms being elements of $G$, horizontal morphisms being elements of $H$, such that there are no non-identity squares in $X$. Then $\cod{X} = G \ast H$ is the free product of the groups $G,H$.
\end{pozn}


\subsection{Lax morphism classifiers}\label{subsekce_laxmors}\vphantom{.}


\begin{defi}\label{defi_resolution}
Let $(T,m,i)$ be a 2-monad on a 2-category $\ck$ and let $(A,a)$ be a strict $T$-algebra. By its \textit{resolution}, denoted $\res{A,a}$, we mean the following coherence data in $\talgs$:
\[\begin{tikzcd}
	{T^3A} && {T^2A} && TA
	\arrow["{Ti_A}"{description}, from=1-5, to=1-3]
	\arrow["Ta"{description}, shift right=4, from=1-3, to=1-5]
	\arrow["{m_A}"{description}, shift left=4, from=1-3, to=1-5]
	\arrow["{m_{T^2A}}"{description}, shift left=4, from=1-1, to=1-3]
	\arrow["{Tm_{TA}}"{description}, from=1-1, to=1-3]
	\arrow["{T^2a}"{description}, shift right=4, from=1-1, to=1-3]
\end{tikzcd}\]
\end{defi}

\begin{theo}
Assume the 2-category $\talgs$ admits lax codescent objects of resolutions of strict algebras. Then the inclusion 2-functor $\talgs \to \talgl$ admits a left 2-adjoint:
\[\begin{tikzcd}
	\talgs && \talgl
	\arrow[""{name=0, anchor=center, inner sep=0}, curve={height=24pt}, hook, from=1-1, to=1-3]
	\arrow[""{name=1, anchor=center, inner sep=0}, "{(-)'}"', curve={height=24pt}, dashed, from=1-3, to=1-1]
	\arrow["\dashv"{anchor=center, rotate=-90}, draw=none, from=1, to=0]
\end{tikzcd}\]
This left adjoint is given by the formula:
\[
(A,a)' = \cod{\res{A,a}}.
\]
\end{theo}
\begin{proof}
Lemma 3.2 in \cite{codobjcoh}.
\end{proof}

Given a strict $T$-algebra $(A,a)$, the algebra $(A',a') := (A,a)'$ has the property of being a \textit{lax morphism classifier}: there is a lax $T$-morphism $(A,a) \rightsquigarrow (A',a')$ (the unit of the above adjunction) such that for any lax morphism $(F,\overline{F}): (A,a) \to (B,b)$ there exists a unique strict $T$-morphism $G: (A',a') \to (B,b)$ such that the following commutes:
\[\begin{tikzcd}
	A && B \\
	\\
	{A'}
	\arrow[squiggly, from=1-1, to=3-1]
	\arrow["{\forall(F,\overline{F})}", squiggly, from=1-1, to=1-3]
	\arrow["{\exists ! G}"', dashed, from=3-1, to=1-3]
\end{tikzcd}\]

\begin{defi}
A monad $(T',m',i')$ on a category $\ce$ is said to be \textit{cartesian} if $T'$ preserves pullbacks and the naturality squares for $m',i'$ are pullbacks.
\end{defi}

If $T': \ce \to \ce$ is a cartesian monad on a category $\ce$ with pullbacks, it preserves internal categories in $\ce$ and thus induces a cartesian 2-monad $\cat(T')$ on the 2-category $\cat(\ce)$ of internal categories and functors\footnote{See for example \cite[Remark 3.16]{johnphd}.}.

We will make use of the following alternative definition of a discrete opfibration. If $A$ is a category, by $s: A_1 \to A_0$ we mean the function that sends a morphism in $A$ to its domain.


\noindent \textbf{Assumption:} Let $(T,m,i)$ be a 2-monad on $\cat$ of form $\cat(T')$ for a cartesian monad $T'$ on $\set$. Assume that $T$ preserves codescent objects. Denote by $U: \talgs \to \cat$ the forgetful 2-functor.

\begin{prop}\label{prop_resAaTiscoddiscr}
$U\res{A,a}$ is a double category. Its transpose, $U\res{A,a}^T$, is codomain-discrete.
\end{prop}
\begin{proof}
The fact that $U\res{A,a}$ is a category internal in $\cat$ follows directly from the fact that $T$ is a cartesian 2-monad. Denote now by $s,t: A_1 \to A_0$ the domain, codomain maps of the category $A$. Regarding $U\res{A,a}$ as a diagram in $\set$:
\[\begin{tikzcd}
	{T^2A_1} && {TA_1} \\
	\\
	{T^2A_0} && {TA_0}
	\arrow["{Ta_1}"', shift right=2, from=1-1, to=1-3]
	\arrow["Tt", shift left=2, from=1-3, to=3-3]
	\arrow["{T^2s}"', shift right=2, from=1-1, to=3-1]
	\arrow["{Ta_0}"', shift right=2, from=3-1, to=3-3]
	\arrow["{T^2t}", shift left=2, from=1-1, to=3-1]
	\arrow["Ts"', shift right=2, from=1-3, to=3-3]
	\arrow["{m_{A_1}}", shift left=2, from=1-1, to=1-3]
	\arrow["{m_{A_0}}", shift left=2, from=3-1, to=3-3]
\end{tikzcd}\]
The domain functor $d_0^T$ of $U\res{A,a}^T$ has object and morphism components these maps:
\begin{align*}
(d_0^T)_0 &= Tt \\
(d_0^T)_1 &= T^2t.
\end{align*}
To show that it's a discrete opfibration is to show that the square below is a pullback in $\set$ (see Example \ref{prbofib}):
\[\begin{tikzcd}
	{T^2A_1} && {TA_1} \\
	\\
	{T^2A_0} && {TA_0}
	\arrow["{T^2t}"', from=1-1, to=3-1]
	\arrow["Tt", from=1-3, to=3-3]
	\arrow["{m_{A_1}}", from=1-1, to=1-3]
	\arrow["{m_{A_0}}"', from=3-1, to=3-3]
\end{tikzcd}\]
This follows from the fact that $m$ is a cartesian natural transformation.
\end{proof}

\begin{nota}
Given a strict $T$-algebra $(A,a)$ we denote:
\[
\cnr{A,a} := \cnr{U\res{A,a}^T}.
\]
\end{nota}

Since by our assumption $T$ preserves codescent objects, we can lift the codescent object $\cnr{A}$ from $\cat$ to $\talgs$. Combining Proposition \ref{prop_resAaTiscoddiscr} with Proposition \ref{propcodinvariantwrttransp}, we obtain:

\begin{theo}\label{THM_laxmorclassifier_is_cnr}
Let $(T,m,i)$ be a 2-monad on $\cat$ of form $\cat(T')$ for  a cartesian monad $T'$ on $\set$. Assume that $T$ preserves reflexive codescent objects. Then the lax morphism classifier for a $T$-algebra $(A,a)$ is given by the category of corners associated to the transpose of the resolution of this $T$-algebra. In other words:
\[
(A,a)' = \cnr{A,a}.
\]
\end{theo}

\subsection{Examples}

\begin{pr}
Let $(T,m,i)$ be the free strict monoidal category 2-monad on $\cat$. It is of form $\cat(T')$ for the free monoid monad on $\set$, and it preserves reflexive codescent objects since it preserves sifted colimits (see \cite[Example 4.3.7]{internalalg}).

Let $(\ca,\otimes)$ be a strict monoidal category. The double category $\res{A,\otimes}$ has:
\begin{itemize}
\item Objects the tuples of objects $(a_1,\dots, a_n) \in \ob T\ca$,
\item vertical morphisms being tuples of morphisms $(f_1,\dots,f_n) \in \mor T\ca$,
\item horizontal morphisms being \textit{partial evaluations}\footnote{See for instance \cite{parteval}.}, that is, objects of $T^2\ca$ whose codomain is given by $T\otimes$ and whose domain is given by the multiplication $m_\ca$. For instance:
\[\begin{tikzcd}
	{(a_1,a_2,a_3,a_4)} &&& {(a_1,a_2\otimes a_3, I, a_4)}
	\arrow["{((a_1),(a_2,a_3),(),(a_4))}", from=1-1, to=1-4]
\end{tikzcd}\]
\item squares being morphisms of $T^2\ca$.
\end{itemize}

The fact that the transpose of this double category is codomain-discrete amounts to having a unique filler for every bottom-left corner in $\res{\ca,\otimes}$, for example consider the following:
\[\begin{tikzcd}
	{(a_1,a_2,a_3,a_4)} \\
	{(b_1,b_2,b_3,b_4)} &&& {(b_1, b_2\otimes b_3, I, b_4)}
	\arrow["{(f_1,f_2,f_3,f_4)}"', from=1-1, to=2-1]
	\arrow["{((b_1),(b_2,b_3),(),(b_4))}"', from=2-1, to=2-4]
\end{tikzcd}\]
The unique filler is given by square $((f_1),(f_2,f_3),(),(f_4)) \in \mor T^2\ca$.

By Theorem \ref{THM_laxmorclassifier_is_cnr}, the lax morphism classifier is the category $\cnr{\ca,\otimes}$ described as follows. The objects are tuples of objects from $\ca$, while a morphism is a tuple $(e,(f_1,\dots, f_n))$ of a partial evaluation followed by a tuple of morphisms. For instance here is an example of a morphism $(a_1,a_2,a_3) \to (b_1,b_2,b_3,b_4)$:
\[\begin{tikzcd}
	{(a_1,a_2,a_3)} &&& {(a_1 \otimes a_2, I, a_3,I)} \\
	&&& {(b_1, b_2, b_3, b_4)}
	\arrow["{((a_1,a_2),(),(a_3),())}", from=1-1, to=1-4]
	\arrow["{(f_1, f_2, f_3, f_4)}", from=1-4, to=2-4]
\end{tikzcd}\]
The strict monoidal structure is given by concatenation of lists. By Lemma \ref{lemma_coddiscrimpliessfs}, this category admits a strict factorization system given by partial evaluations followed by tuples of morphisms of $\ca$.
\end{pr}

The following is the extension of the previous example in the sense that we obtain it if we put $X = *$:

\begin{pr}\label{prcodlax-colaxfunclass}
Fix a set $X$. Consider a cartesian monad $T$ on $\set^{X \times X}$ (the category of graphs with the set of objects being $X$) given by paths:
\[
(T\cc)_{A,B} = \text{Path}_\cc(A,B) = \{ (f_1,\dots, f_m) | m \in \mathbb{N}, \text{cod}(f_{i})= \text{dom}(f_{i+1}) \forall i < m \}
\]

\noindent Consider its extension $\cat(T)$ (that we again denote by $T$) to $\cat(\set^{X \times X}) = \cat^{X \times X}$, the 2-category of $\cat$-graphs whose set of objects is $X$. A strict $T$-algebra $\cc$ is a $\cat$-graph equipped with a composition functor for each tuple $(A,B) \in X\times X$:
\[
\text{Path}_\cc(A,B) \to \cc(A,B).
\]
It is easily verified that such $\cc$ is precisely a small 2-category with the set of objects being $X$. Also, lax $T$-algebra morphisms are identity-on-objects lax functors.

Any 2-category $\cc$ (regarded as a $T$-algebra) gives rise to its resolution $U\res{\cc}$, which is a diagram in $\cat^{X \times X}$. Denote its codescent object by $\cc'$ - this is a $\cat$-graph with the set of objects being $X$ that moreover has the structure of a 2-category.

As colimits in $\cat^{X \times X}$ are computed pointwise, $\cc'(x,y)$ is given by the codescent object of $\res{\cc}(x,y)$. Because the 2-monad multiplication $m: T^2 \Rightarrow T$ is a pointwise discrete opfibration, each $\res{\cc}(x,y)$ is a codomain-discrete double category and so $\cc'(x,y)$ can be computed using the category of corners construction as follows. 

Objects in $\cc'(x,y)$ are the objects of $T\cc(x,y)$, that is, paths of morphisms in the 2-category $\cc$. Morphisms are corners whose first component is given by a \textit{partial evaluation 2-cell} (an object of $T^2\cc(x,y)$) and the second component is given by a tuple of 2-cells in $\cc$ (a morphism in $T\cc(x,y)$). For instance this morphism $(f_1, f_2, f_3,f_4) \to (g_1, g_2, g_3)$:
\[\begin{tikzcd}
	{a_1} && {a_2} && {a_3} && {a_4} && {a_5} \\
	&&& {a_3} && {a_4}
	\arrow[""{name=0, anchor=center, inner sep=0}, Rightarrow, no head, from=2-6, to=1-9]
	\arrow[""{name=1, anchor=center, inner sep=0}, "{f_2 \circ f_1}"{pos=0.7}, from=1-1, to=2-4]
	\arrow[""{name=2, anchor=center, inner sep=0}, "{g_1}"', curve={height=30pt}, from=1-1, to=2-4]
	\arrow[""{name=3, anchor=center, inner sep=0}, "{g_3}"', curve={height=30pt}, from=2-6, to=1-9]
	\arrow["{f_1}", from=1-1, to=1-3]
	\arrow["{f_2}", from=1-3, to=1-5]
	\arrow["{f_4}", from=1-7, to=1-9]
	\arrow["{f_3}", from=1-5, to=1-7]
	\arrow[""{name=4, anchor=center, inner sep=0}, "{g_2}"', curve={height=30pt}, from=2-4, to=2-6]
	\arrow[""{name=5, anchor=center, inner sep=0}, "{f_4 \circ f_3}"{description}, from=2-4, to=2-6]
	\arrow["{\alpha_1}", shorten <=4pt, shorten >=4pt, Rightarrow, from=1, to=2]
	\arrow["{\alpha_3}", shorten <=4pt, shorten >=4pt, Rightarrow, from=0, to=3]
	\arrow["{\alpha_2}", shorten <=4pt, shorten >=4pt, Rightarrow, from=5, to=4]
	\arrow["{((f_1,f_2),(f_3,f_4),())}"{pos=0.4}, shorten <=3pt, shorten >=3pt, Rightarrow, from=1-5, to=5]
\end{tikzcd}\]
The 2-category structure of the lax functor classifier $\cc'$ is given by concatenation of paths and tuples of 2-cells. By Lemma \ref{lemma_coddiscrimpliessfs}, each hom category of $\cc'$ admits a strict factorization system given by \textit{partial evaluation 2-cells} and tuples of 2-cells on $\cc$. Moreover, these strict factorization systems are stable under post- and pre-composition with 1-cells of $\cc'$.


This description of the lax functor classifier 2-category has been sketched in \cite[Page 246]{elephant}.
\end{pr}

\begin{pozn}[Colax morphism classifiers]
We can apply dualities to compute \textbf{colax} morphism classifier for a 2-monad $T$ on $\cat$ of form $\cat(T')$ as follows. First note that the opposite category 2-functor $(-)^{op}: \cat^{co} \to \cat$ induces a 2-isomorphism
\begin{align*}
\talgc &\cong T^{co}\text{-Alg}_l,\\
(A,a) &\mapsto (A^{op},a^{op}).
\end{align*}
This implies that a $T$-algebra $(B,b)$ is the colax $T$-morphism classifier for $(A,a)$ if and only if $(B^{op},b^{op})$ is the lax $T^{v}$-morphism classifier for $(A^{op},a^{op})$.

Now let $(A,a)$ be a strict $T$-algebra. The lax-$T^{co}$-morphism classifier for $(A^{op},a^{op})$ is a $T^{co}$-algebra $\cnr{A^{op}}$, and thus the colax morphism classifier is given by the formula:
\[
(A,a)' = \cnr{A^{op}}^{op}.
\]

For instance, the colax monoidal functor classifier for a strict monoidal category $(\ca,\otimes)$ again has tuples of objects in $\ca$ as object, and a morphism $(a_1,a_2,a_3,a_4) \to$\\ $\to (b_1,b_2,b_3,b_4)$ is a corner (or rather, a \textit{cospan}) like this:
\[\begin{tikzcd}
	{(a_1,a_2,a_3,a_4)} \\
	{(b_1, b_2\otimes b_3, I, b_4)} &&& {(b_1,b_2,b_3,b_4)}
	\arrow["{(f_1,f_2,f_3,f_4)}"', from=1-1, to=2-1]
	\arrow["{((b_1),(b_2,b_3),(),(b_4))}", from=2-4, to=2-1]
\end{tikzcd}\]
\end{pozn}

\begin{pr}\label{prcodlax-dblcats}
While this example does not follow from the results as stated in this paper, it follows from their internal analogue in $\cat(\ce)$, where $\ce = \gph$. Let $\fc : \gph \to \gph$ be the free category on a graph monad. It is a cartesian monad, it thus admits an extension to a 2-monad $T:= \cat(\fc)$ on the 2-category $\cat(\gph)$. This 2-category consists of structures like double categories except we can not compose squares or horizontal morphisms horizontally, there's only vertical composition.

A strict $T$-algebra is a double category. It can also be seen that a colax $T$-algebra morphism is a colax double functor. Given a $T$-algebra $X$, the construction $\cnr{X}$ of the colax double functor classifier  agrees with the construction $\mathbb{P}\text{ath } X$ of \cite[The construction 1.1, Proposition 1.19]{pathsindblcats}. The fact that it admits an internal strict factorization system was also proven in \cite[1.5 Proposition]{pathsindblcats}.

\end{pr}

The internal versions of the results in this paper as well as the generalization of the category of corners to lax $T$-algebras will appear in the author's upcoming Ph.D. thesis.

\bibliographystyle{plain} 
\bibliography{ClanekREF} 


\end{document}